\documentclass[12pt]{amsart}

\usepackage{amsmath, amssymb}
\usepackage{color,url}

\newtheorem{theorem}{Theorem}[section]

\newtheorem{prop}[theorem]{Proposition}

\theoremstyle{definition}
\newtheorem{defi}[theorem]{Definition}

\theoremstyle{remark}

\theoremstyle{remark}
\newtheorem{remark}[theorem]{Remark}

\numberwithin{equation}{section}

\newcommand{\PP}{\mathbb{P}}
\newcommand{\RR}{\mathbb{R}}

\newcommand{\rat}{{\mathbb Q}}

\newcommand{\cov}{\operatorname{Cov}}
\newcommand{\varr}{\operatorname{Var}}

\newcommand{\dd}{\mathrm{d}}
\newcommand{\ee}{\mathrm{e}}
\newcommand{\cA}{{\mathcal A}}
\newcommand{\cB}{{\mathcal B}}
\newcommand{\cD}{{\mathcal D}}

\newcommand{\cL}{{\mathcal L}}
\newcommand{\EE}{\operatorname{\mathbb{E}}}
\newcommand{\var}{\operatorname{Var}}
\newcommand{\proofend}{\hfill\mbox{$\Box$}}

\newcommand{\distre}{\stackrel{\cD}{=}}

\allowdisplaybreaks

\begin{document}
\sloppy
\title[Gauss-Markov processes as scaled Ornstein-Uhlenbeck processes]
 {Gauss-Markov processes as space-time scaled stationary Ornstein-Uhlenbeck processes}

\author{M\'aty\'as Barczy}
\address{M\'aty\'as Barczy, MTA-SZTE Analysis and Stochastics Research Group,
 Bolyai Institute, University of Szeged, Aradi v\'ertan\'uk tere 1, H--6720 Szeged, Hungary.}
\email{barczy\@@{}math.u-szeged.hu}
\thanks{M\'aty\'as Barczy is supported by the J\'anos Bolyai Research Scholarship of the Hungarian Academy of Sciences.}

\author{Peter Kern}
\address{Peter Kern, Mathematisches Institut, Heinrich-Heine-Universit\"at D\"usseldorf, Universit\"atsstr. 1, D-40225 D\"usseldorf, Germany}
\email{Peter.Kern\@@{}uni-duesseldorf.de}

\date{\today}

\begin{abstract}
We present a class of Gauss-Markov processes which can be represented as space-time scaled
 stationary Ornstein-Uhlenbeck processes defined on the real line.
We give several explicit examples of the representation for certain Gauss bridge processes.
As an application, we derive a formula for the density function of the supremum location of certain standardized
 Gauss-Markov processes on compact time intervals.
We also present some sufficient conditions under which mean centered Gauss-Markov processes take zero at a fixed time
 with probability one.
\end{abstract}

\keywords{Gauss-Markov process, stationary Ornstein-Uhlenbeck process, Wiener bridge, supremum location.}

\subjclass[2010]{Primary 60G15; Secondary 60G10, 60J65}

\maketitle

\baselineskip=18pt

\section{Introduction}

In this paper, we present a class of Gauss-Markov processes which
 can be represented as space-time scaled stationary Ornstein-Uhlenbeck processes defined on the real line
 by specifying the space and time transformations in question explicitly as well.
We will give several examples, and we will use this representation for determining the distribution
 of the supremum location of certain standardized Gauss-Markov processes on compact time intervals.
To motivate our method, we first recall the well-known facts that a Wiener bridge can be represented
 as a space-time scaled stationary Ornstein-Uhlenbeck process (see, e.g., Shorack and Wellner \cite[Exercise 10, page 32]{ShoWel}) and
 that the supremum location of a Wiener bridge on the interval $[0,1]$ is uniformly distributed (see, e.g., Ferger \cite[Section 2]{Fer}).
Let $(B_t)_{t\geq 0}$ be a standard Wiener process, then its Lamperti transform (see Lamperti \cite[page 64]{Lam})
 \begin{align}\label{OU_transform}
    S_t:=\ee^{-\frac{t}{2}} B_{\ee^t},\qquad t\in\RR,
 \end{align}
 defines a strictly stationary centered Gauss process $S=(S_t)_{t\in\RR}$ defined on the real line with
 \begin{align}\label{cov_stac_OU}
   \cov(S_s,S_t) = \ee^{-\frac{\vert t-s\vert}{2}}, \qquad s,t\in\RR,
 \end{align}
 see, e.g., Doob \cite{Doo} or Shorack and Wellner \cite[Exercise 9, page 32]{ShoWel}.
The process $S$ is known as a stationary Ornstein-Uhlenbeck process defined on $\RR$.
Then a Wiener bridge from $0$ to $0$ over the time interval $[0,1]$
 generates the same law on $(C([0,1]), \cB(C([0,1])))$ as the space-time scaled stationary Ornstein-Uhlenbeck process
 \[
   U_t
    :=\begin{cases}
       \sqrt{t(1-t)}S\left(\ln \left(\frac{t}{1-t}\right)\right)
           & \text{if $t\in(0,1)$},\\
        0  & \text{if \ $t=0$ \ or \ $t=1$,}
     \end{cases}
 \]
 see, e.g., Shorack and Wellner \cite[Exercise 10, page 32]{ShoWel},
 where for a subset $D\subset[0,\infty)$, $C(D)$ denotes the space of continuous real-valued functions defined on $D$
 and $\cB(C(D))$ is the Borel $\sigma$-algebra on $C(D)$.
Recall that the law of the pathwise unique strong solution of the stochastic differential equation (SDE)
 \[
   \dd Z_t = -\frac{1}{1-t} Z_t\,\dd t + \dd B_t,\qquad t\in[0,1),
 \]
 with an initial value $Z_0=0$ coincides with that of the Wiener bridge from $0$ to $0$ over the time interval $[0,1]$.

As a generalization of the observation above, first in Section \ref{section_gen_fram} we provide a class of Gauss-Markov processes
 (satisfying a linear SDE) which can be represented as space-time scaled stationary Ornstein-Uhlenbeck processes defined
 on the real line by specifying the space and time transformations in question explicitly, see Theorem \ref{Thm_General}.
Then we derive a formula for the density function of the supremum location of the standardized version of the Gauss-Markov processes
  in question on compact time intervals, see Section \ref{section_application}.
A partial converse of Theorem \ref{Thm_General} is given in Proposition \ref{Prop_General_reversed}.
Further, we present some sufficient conditions under which the mean centered Gauss-Markov process in question take zero at a fixed time
 with probability one, see Proposition \ref{Prop_General}.
In Section \ref{sec_examples}, we give some examples: scaled Wiener bridges (also called general $\alpha$-Wiener bridges),
 Ornstein-Uhlenbeck type bridges, weighted Wiener bridges and so called $F$-Wiener bridges.
In Section \ref{section_counterexamples} we present some counterexamples where the representation given in Theorem \ref{Thm_General} does not hold
 such as the zero area Wiener bridge.
All the proofs are presented in Appendix \ref{App_proofs}.

\section{A general framework}\label{section_gen_fram}

In what follows, let $\rat$ and $\RR_+$ denote the set of rational and non-negative real numbers, respectively.
For $s,t\in\RR$, let $s\wedge t$ denote $\min(s,t)$, and let $\cB(\RR)$ denote the set of Borel sets of $\RR$.
Recall that $C([0,T])$ with $T\in(0,\infty)$, and $C([0,\infty))$ are complete, separable metric spaces
 (with appropriate metrics).
Due to the strictly increasing and continuous time change $\frac{2T}{\pi}\arctan s$, $s\in[0,\infty)$
 (which is a bijection between $[0,\infty)$ and $[0,T)$), we get $C([0,T))$ is a
 complete, separable metric space as well.
For a stochastic process $X$, the time dependence of $X$ will be denoted by $X_t$
 or $X(t)$ depending on how complicated (large sized) the expression replacing $t$ is.

Let $T\in (0,\infty]$.
Let $\phi:[0,T)\to (0,\infty)$ be a continuously differentiable function with $\phi(0)=1$,
 $\psi, \sigma:[0,T) \to\RR$ be continuous functions, and
 suppose that $\sigma(t)\ne 0$ on some interval $(0,\delta)$ for some $\delta\in(0,T]$.
Let us consider the SDE
 \begin{align}\label{SDE_general}
   \dd Z_t = \left(\frac{\phi'(t)}{\phi(t)} Z_t + \psi(t)\right)\dd t + \sigma(t) \dd B_t, \qquad t\in[0,T),
 \end{align}
 with a non-random initial value $Z_0=\xi\in\RR$, where $(B_t)_{t\in\RR_+}$ is a standard Wiener process
 on a probability space $(\Omega,\cA,\PP)$.
Note that in the drift coefficient of the SDE \eqref{SDE_general} the factor $\phi'(t)/\phi(t)$ can be an arbitrary
 continuous function $f:\RR_+\to\RR$, since the Cauchy problem $\frac{\phi'(t)}{\phi(t)}= f(t)$, $t\in\RR_+$,
 with $\phi(0)=1$ has the unique solution
 \[
   \phi(t) = \exp\left\{\int_0^t f(u)\,\dd u\right\}, \qquad t\in\RR_+.
 \]
However, we keep the form $\phi'(t)/\phi(t)$ in the SDE \eqref{SDE_general} in order to have a more compact presentation.

Since $\phi'/\phi$, $\psi$ and $\sigma$ are non-random, measurable and locally bounded, by using It\^{o}'s formula,
 \begin{align}\label{gen_SDE_solution}
  Z_t = \phi(t) \left(\xi + \int_0^t\frac{\psi(u)}{\phi(u)}\,\dd u
                      + \int_0^t \frac{\sigma(u)}{\phi(u)}\,\dd B_u\right),\qquad t\in[0,T),
 \end{align}
 can be shown to be the pathwise unique strong solution of the SDE \eqref{SDE_general}.
The Gauss-Markov process $Z$ is called a process of Ornstein-Uhlenbeck type with parameters
 $\phi$, $\psi$ and $\sigma$ in Patie \cite[pages 49 and 58]{Pat}.
We also note that processes of the form \eqref{gen_SDE_solution} have been recently applied
 in Alili and Patie \cite[second proof of Theorem 1.2]{AliPat}.

One can easily calculate
 \begin{align*}
   \cov(Z_s,Z_t)
    = \phi(s)\phi(t)\int_0^{s\wedge t} \frac{\sigma(u)^2}{\phi(u)^2}\,\dd u,
    \qquad s,t\in[0,T).
 \end{align*}

Let us consider the mean centered process defined by
 \begin{align}\label{help_centered}
   \widetilde Z_t
           := Z_t - \EE(Z_t)
            = \phi(t)\int_0^t \frac{\sigma(u)}{\phi(u)}\,\dd B_u,
                 \qquad t\in[0,T).
 \end{align}

As a consequence of Theorem 4.1 in Lachout \cite{Lac1} we have the following representation theorem.
For completeness, in Appendix \ref{App_proofs} we present two proofs of it: the first one
 is based on Theorem 4.1 in Lachout \cite{Lac1} by checking its conditions, the second one is a direct one based
 on Dambis, Dubins and Schwarz lemma.

\begin{theorem}\label{Thm_General}
There exists a strictly stationary centered Ornstein-Uhlenbeck process $R=(R_t)_{t\in\RR}$
 with $\cov(R_s,R_t) = \ee^{-\frac{\vert t-s\vert}{2}}$, $s,t\in\RR$, such that
 \begin{align}\label{help_main_transform}
     \widetilde Z_t
          = \phi(t)\sqrt{\int_0^t \frac{\sigma(u)^2}{\phi(u)^2}\,\dd u}\,
            \,R\left(\ln\left( \int_0^t \frac{\sigma(u)^2}{\phi(u)^2}\,\dd u \right)\right),
            \qquad  \forall\; t\in[0,T) \quad \text{a.s.,}
 \end{align}
 where $\widetilde Z$ is defined in \eqref{help_centered}, and
 the right hand side of \eqref{help_main_transform} for $t=0$ is understood as an almost sure
 limit as $t\downarrow 0$.
Roughly speaking, the mean centered process $(\widetilde Z_t)_{t\in[0,T)}$ coincides
 with a space-time scaled stationary Ornstein-Uhlenbeck process almost surely.
Further, if $\PP(\widetilde Z_T:=\lim_{t\uparrow T} \widetilde Z_t = 0)=1$, then
 \begin{align}\label{rep_closed}
     \widetilde Z_t
          = \phi(t)\sqrt{\int_0^t \frac{\sigma(u)^2}{\phi(u)^2}\,\dd u}\,
            \,R\left(\ln\left( \int_0^t \frac{\sigma(u)^2}{\phi(u)^2}\,\dd u \right)\right),
            \qquad \forall\; t\in[0,T] \quad \text{a.s.,}
 \end{align}
 where the right hand side of the above equation for $t=0$ and for $t=T$ is understood
 as an almost sure limit as $t\downarrow 0$ and $t\uparrow T$, respectively.
\end{theorem}

\begin{remark}\label{REMARK1}
(i) Using the notations of Theorem \ref{Thm_General}, the standardized version of the process $Z$ can be
 represented as
 \[
  {Z^*_t:=}
     \frac{Z_t - \EE(Z_t)}{\sqrt{\var(Z_t)}}
     = R\left(\ln\left( \int_0^t \frac{\sigma(u)^2}{\phi(u)^2}\,\dd u \right)\right),
      \qquad  \forall\; t\in(0,T) \quad \text{a.s.}
 \]
\noindent{(ii)}
Using the notations of Theorem \ref{Thm_General}, if we additionally suppose that
 $\lim_{t\uparrow T} \phi(t)\sqrt{\int_0^t \frac{\sigma(u)^2}{\phi(u)^2}\,\dd u} =0$,
 then, since $R$ is strictly stationary, by Slutsky's lemma, we have
 \[
     \phi(t)\sqrt{\int_0^t \frac{\sigma(u)^2}{\phi(u)^2}\,\dd u}\,
            \,R\left(\ln\left( \int_0^t \frac{\sigma(u)^2}{\phi(u)^2}\,\dd u \right)\right)
 \]
 converges in probability to $0$ as $t\uparrow T$.
Later, under some stronger additional assumptions, we will strengthen this statement, namely,
 instead of convergence in probability we will prove almost sure convergence, see Proposition \ref{Prop_General}.
\proofend
\end{remark}

Now, we formulate a partial converse of Theorem \ref{Thm_General}.

\begin{prop}\label{Prop_General_reversed}
For some $T\in(0,\infty]$, let $a:(0,T)\to(0,\infty)$ and $b:(0,T)\to\RR$ be continuously differentiable functions
 such that $b'(t)>0$ for all $t\in(0,T)$, and $\lim_{t\downarrow 0} b(t) = -\infty$.
Further, let $\psi:[0,T)\to\RR$ be a continuous function, and
  \begin{align}\label{help_converse}
    \phi(t):=\frac{a(t)}{\ee^{b(t)/2}}, \quad t\in(0,T),
    \qquad \text{and}\qquad \sigma(t):=\sqrt{b'(t)a^2(t)},\quad t\in(0,T).
  \end{align}
Provided that $\phi$ can be extended to be continuously differentiable on $[0,T)$
 with $\phi(0)=1$, and $\sigma$ can be extended to be continuous on $[0,T)$,
 there exists a strictly stationary centered Ornstein-Uhlenbeck process $R=(R_t)_{t\in\RR}$
 with $\cov(R_s,R_t) = \ee^{-\frac{\vert t-s\vert}{2}}$, $s,t\in\RR$, such that
 \begin{align}\label{help_main_transform_reversed}
      a(t) R_{b(t)} = \widetilde Z_t = Z_t - \EE(Z_t),
            \qquad  \forall\; t\in[0,T) \quad \text{a.s.,}
 \end{align}
 where $(Z_t)_{t\in[0,T)}$ is the pathwise unique strong solution of the SDE \eqref{SDE_general}
 with the extended versions of $\phi$ and $\sigma$ given in \eqref{help_converse} and
 with a non-random initial value $Z_0=\xi\in\RR$, and the left hand side of \eqref{help_main_transform_reversed} for $t=0$ is understood
 as an almost sure limit as $t\downarrow 0$.
\end{prop}

Next, we formulate some sufficient conditions under which $\widetilde Z_t$ converges to $0$ as $t\uparrow T$ with probability one.

\begin{prop}\label{Prop_General}
If $\lim_{t\uparrow T} \phi(t)=0$ and there exists some $\varepsilon>0$ such that the function
 \begin{align}\label{help_boundedness}
  \phi(t)\left( \int_0^t \frac{\sigma(u)^2}{\phi(u)^2}\,\dd u \right)^{\frac{1}{2}+\varepsilon},
   \qquad t\in[0,T),
 \end{align}
 is bounded, then for the mean centered process $(\widetilde Z_t)_{t\in[0,T)}$ we have
 $\PP(\widetilde Z_T:=\lim_{t\uparrow T} \widetilde Z_t = 0)=1$.
\end{prop}

We note that Li \cite[Proposition 1]{Li} and Hildebrandt and Roelly \cite[Proposition 2.3]{HilRoe} derived
 other sufficient conditions under which $\PP(\lim_{t\uparrow T} \widetilde Z_t = 0)=1$ holds.

\section{Distribution of the supremum location}\label{section_application}

For a stochastic process $X=(X_t)_{t\in\RR}$ with continuous sample paths and
 for an interval $[a,b]$, $a<b$, $a,b\in\RR$, let
 \begin{align*}
     \tau_{X,[a,b]}:=\inf\big\{ t\in[a,b] : X_t=\sup_{s\in[a,b]} X_s\big\} \qquad \text{and} \qquad
      M_{X,[a,b]}:=\sup_{s\in[a,b]} X_s,
 \end{align*}
 i.e., $\tau_{X,[a,b]}$ is the leftmost time point at which the overall supremum $M_{X,[a,b]}$ in the interval $[a,b]$ is achieved.
In what follows, we will simply call $\tau_{X,[a,b]}$ the supremum location for $X$ on the interval $[a,b]$.
We mention that the almost sure uniqueness of the supremum location on compact intervals of $\RR_+$
 for a continuous Gauss process $X$ satisfying $\var(X_t-X_s)\ne 0$ for $s\ne t$, $s,t\in\RR_+$,
 was proved in Kim and Pollard \cite[Lemma 2.6]{KimPol}.
For a corresponding result for regular one-dimensional diffusions (in the sense of It\^{o}--McKean \cite{ItoMck}), see Cs\'aki et al.\ \cite[part (i)
 of Remarks (2.2)]{CsaFolSal}.
Recently, Pimentel \cite[Theorem 1]{Pim} has given necessary and sufficient conditions for the almost sure uniqueness of the
 supremum location on compact intervals for stochastic processes with continuous sample paths,
 and then L\'opez and Pimentel \cite{LopPim} extended it to a variety of non-continuous and multivariate processes.

As it was mentioned in the Introduction, the supremum location of a Wiener bridge on the interval $[0,1]$ is uniformly distributed
 (see, e.g., Ferger \cite[Section 2]{Fer}).
The supremum location of the absolute value of a Wiener bridge on the interval $[0,1]$ is also absolutely continuous, Ferger \cite{Fer2} determined its
 density function.
Very recently, Ferger \cite[Theorem 1.1]{Fer3} has determined the density function of a Wiener bridge on $[0,1]$ standardized by its supremum location.
 Alabert and Caballero \cite[Proposition 3.1]{AlaCab} calculated the probability that the infimum location of the concatenation of
 Wiener bridges is located in one of the time intervals of the Wiener bridges building up the concatenation.

Using the notations of Section \ref{section_gen_fram}, if the set $\{t\in[0,T) : \sigma(t) = 0\}$ does not contain any interval and
 $\lim_{t\uparrow T}\int_0^t \frac{\sigma(u)^2}{\phi(u)^2}\,\dd u =: S_T\in(0,\infty]$,
 then the continuous function $\beta_T:(0,T)\to (-\infty,\ln( S_T))$,
 \[
   \beta_T(t):=\ln\left(\int_0^t \frac{\sigma(u)^2}{\phi(u)^2}\,\dd u \right),\qquad  t\in(0,T),
 \]
 is strictly increasing having inverse $\beta_T^{-1}:(-\infty,\ln( S_T))\to(0,T)$.
Consequently, using that $R$ is strictly stationary, by part (i) of Remark \ref{REMARK1}, we have
 \begin{align}\label{help_taubeta}
  \tau_{Z^*,[t_1,t_2]}
    = \beta_T^{-1}(\tau_{R,[ \beta_T(t_1),\beta_T(t_2)]})
   \distre
     \beta_T^{-1}\big( \beta_T(t_1) + \tau_{R,[0,\beta_T(t_2)-\beta_T(t_1)]}\big)
 \end{align}
 for any $0<t_1<t_2<T$, where $Z^*$ is given in part (i) of Remark \ref{REMARK1}, the first equality holds almost surely
 and $\distre$ denotes equality in distribution.
So we reduced the problem of determining the distribution of the supremum location of the standardized version $Z^*$
 to determine the distribution of the supremum location of $R$ on compact subintervals of the form $[0,T]$, $T>0$.

Samorodnitsky and Shen \cite{SamShe2} provide a lot of information on the distribution
 of $\tau_{R,[0,T]}$, $T>0$.
By Theorem 3.1 in Samorodnitsky and Shen \cite{SamShe2}, the restriction of the law of
 $\tau_{R,[0,T]}$ to the interior $(0,T)$ of the interval $[0,T]$ is absolutely continuous and very specific
 properties of the density function in question have been described, e.g.,
 the (c\`{a}dl\`{a}g version of) the density function is bounded away from 0 and its value
 at $t$ is bounded above by $\max(\frac{1}{t},\frac{1}{T-t})$,
 $t\in(0,T)$.
Using that $R$ is time reversible, i.e., the laws of the processes $(R_{-t})_{t\in\RR}$ and $(R_t)_{t\in\RR}$
 coincide, a finer upper bound for the density function in question has been derived, see
 Samorodnitsky and Shen \cite[Proposition 4.2]{SamShe2}.
Further, this density function is not bounded near each of the two endpoints of the interval $[0,T]$,
 and $\PP(\tau_{R,[0,T]}=0) = \PP(\tau_{R,[0,T]} = T) = 0$, $T>0$, see Samorodnitsky and Shen \cite[Example 3.7]{SamShe2}.
Later it will turn out that the law of $\tau_{R,[0,T]}$ (without restriction to $(0,T)$) is absolutely continuous,
 and we will derive an expression for its density function as well.
We point out that, compared to the general setup of Samorodnitsky and Shen \cite{SamShe2}, we can take the
 advantage that $R$ is not only strictly stationary, but a time-homogeneous Markov process as well,
 and hence we can use a general result of Cs\'aki et al. \cite{CsaFolSal} to handle the distribution of the supremum
 location of $R$.
Very recently, Shen \cite[Theorem 4.3]{She} has proved that for the infimum location on the interval $[0,1]$ of a non-constant, self-similar L\'evy process,
 one of the following three scenarios is true: it is concentrated at $0$ with probability one; it is concentrated at $1$ with probability one;
 or it is Beta-distributed with some given parameters.

In what follows we present a procedure which results in a formula for
 a density function of  the distribution of the supremum location of $R$
 on a compact interval of the form $[0,T]$, $T\in\RR_+$, \ and that of $Z^*$
 on a compact interval of the form $[t_1,t_2]$, $0<t_1<t_2<T$.
Note that the supremum location of $R$ in question is unique almost surely, see, Kim and Pollard \cite[Lemma 2.6]{KimPol}.
First, recall that the law of $(R_t)_{t\in\RR_+}$ can be represented as the law of the pathwise unique strong solution
 of an appropriate SDE.
Namely, if $(B_t)_{t\in\RR_+}$ is a standard Wiener process and $\xi$ is a standard normally distributed random variable
 independent of $(B_t)_{t\in\RR_+}$, then the process
 \[
   V_t:=\ee^{-\frac{t}{2}} \left( \xi + \int_0^t\ee^{\frac{r}{2}}\,\dd B_r\right),\qquad t\in\RR_+,
 \]
 is the pathwise unique strong solution of the SDE
 \[
   \dd V_t = -\frac{1}{2} V_t\,\dd t + \dd B_t, \qquad t\in\RR_+,
 \]
 with initial condition $V_0 = \xi$, and $(V_t)_{t\in\RR_+}$
 generates the same measure on $C(\RR_+)$ as $(R_t)_{t\in\RR_+}$.
The mapping
 \[
   C([0,T])\ni f\mapsto (M_{f,[0,T]}, f(T), \tau_{f,[0,T]}))\in\RR\times\RR\times[0,T]
 \]
 is measurable for all $T>0$, since $\{\tau_{f,[0,T]} \leq t\} =\{ M_{f,[0,t]} \geq M_{f,[t,T]} \}$ for all
 $t\in[0,T]$, $\{f\in C([0,T]) : f(T)\leq u\}$ is a cylinder set for $C([0,T])$ for any $u\in\RR$, and
 \[
   \{ M_{f,[0,T]} > u\} = \bigcup_{\{ s\in \rat \cap [0,T] \}} \{ f(s) > u\}, \qquad u\in\RR.
 \]
Hence the laws of
 $(M_{V,[0,T]}, V_T, \tau_{V,[0,T]})$ and  $(M_{R,[0,T]}, R_T, \tau_{R,[0,T]})$
 coincide for all $T>0$.
Using that the so-called scale function and speed measure
 (see, e.g., Karatzas and Shreve \cite[Section 5, formulae (5.42) and (5.51)]{KarShr}) corresponding to $V$
 take the forms
 \begin{align*}
   &S_c(x)=\int_c^x \exp\left\{ -2\int_c^y -\frac{z}{2}\,\dd z\right\}\dd y
         = \int_c^x \ee^{\frac{y^2-c^2}{2}}\,\dd y, \qquad x\in\RR,\quad c\in\RR,\\
   &m_c(B) = \int_B 2\ee^{-\frac{x^2-c^2}{2}}\,\dd x, \qquad B\in\cB(\RR),\quad  c\in\RR,
 \end{align*}
 by Theorem A in Cs\'aki et al. \cite{CsaFolSal}, we get
 \begin{align}\label{help_Csaki}
 \begin{split}
  &\PP( M_{R,[0,T]}\in A, R_T \in B, \tau_{R,[0,T]}\in C \mid R_0=x)\\
  & \quad = \int_A\int_B\int_C n_x(s,y) n_z(T-s,y) \mathbf 1_{\{x\leq y\}} \mathbf 1_{\{z\leq y\}}\, S_c(\dd y)\, m_c(\dd z)\,\dd s\\
  & \quad = \int_A\int_B\int_C n_x(s,y) n_z(T-s,y) 2\ee^{-\frac{z^2-c^2}{2}} \ee^{\frac{y^2-c^2}{2}}
            \mathbf 1_{\{x\leq y\}} \mathbf 1_{\{z\leq y\}}\, \dd y\,\dd z\,\dd s\\
  &\quad  = \int_A\int_B\int_C n_x(s,y) n_z(T-s,y) 2\ee^{-\frac{z^2-y^2}{2}}
            \mathbf 1_{\{x\leq y\}} \mathbf 1_{\{z\leq y\}} \,\dd y\,\dd z\,\dd s
 \end{split}
 \end{align}
 for all $T>0$, $x\in\RR$, $A,B\in\cB(\RR)$ and $C\in\cB([0,T])$, where
 $(0,\infty)\ni u \mapsto n_x(u,y)$ denotes the (conditional) density function
 of the random variable $\inf\{t\in(0,\infty) : R_t=y\}$ provided that $R_0=x$, where $x,y\in\RR$.
Note that the above formula does not depend on $c\in\RR$.
In Alili et al. \cite{AliPatPed} one can find several formulae for $n_x(u,y)$, $u\in(0,\infty)$, e.g., due to their formula (4.1),
 for all $x<y$, $x,y\in\RR$, and $u\in(0,\infty)$,
 \begin{align*}
  n_x(u,y)
    = \frac{1}{\pi} \int_0^\infty \cos\left(\frac{u\alpha}{2}\right)
       \frac{Hr_{-\alpha}\left(\frac{-y}{\sqrt{2}}\right) Hr_{-\alpha}\left(\frac{-x}{\sqrt{2}}\right)
             + Hi_{-\alpha}\left(\frac{-x}{\sqrt{2}}\right)Hi_{-\alpha}\left(\frac{-y}{\sqrt{2}}\right)}
              {\left(Hr_{-\alpha}\left(\frac{-y}{\sqrt{2}}\right)\right)^2 + \left(Hi_{-\alpha}\left(\frac{-y}{\sqrt{2}}\right)\right)^2}
              \,\dd\alpha,
 \end{align*}
 where, for all $\alpha\in\RR$, the functions
 \begin{align*}
  & Hr_\alpha(v):= \frac{2}{\sqrt{\pi}} \int_0^\infty \ee^{-s^2}
       \cos\left(\frac{\alpha}{2}\log\left(1+\left(\frac{v}{s}\right)^2\right)\right)\dd s,\qquad v\in\RR,\\[1mm]
  & Hi_\alpha(v):= \frac{2}{\sqrt{\pi}} \int_0^\infty \ee^{-s^2}
       \sin\left(\frac{\alpha}{2}\log\left(1+\left(\frac{v}{s}\right)^2\right)\right)\dd s,\qquad v\in\RR,
 \end{align*}
 are the real and imaginary parts of certain normalized Hermite functions, respectively.
Using \eqref{help_Csaki}, by the law of total probability and Fubini's theorem,
 \begin{align*}
  &\PP(\tau_{R,[0,T]}\in C)\\
  & = \int_{-\infty}^\infty \PP( M_{R,[0,T]}\in \RR, R_T \in \RR, \tau_{R,[0,T]}\in C \mid R_0=x)
            \frac{1}{\sqrt{2\pi}} \ee^{-\frac{x^2}{2}}\,\dd x\\
  & = \int_C \left(\int_{-\infty}^\infty\int_{-\infty}^\infty \int_{-\infty}^\infty n_x(s,y) n_z(T-s,y) 2\ee^{-\frac{z^2-y^2}{2}}
            \mathbf 1_{\{x\leq y\}} \mathbf 1_{\{z\leq y\}} \frac{1}{\sqrt{2\pi}} \ee^{-\frac{x^2}{2}}
            \,\dd x \,\dd y\,\dd z\right)\dd s
 \end{align*}
 for all $C\in\cB([0,T])$.
Consequently, a density function of $\tau_{R,[0,T]}$ can be chosen as
 \begin{align}\label{density_R}
   f_{\tau_{R,[0,T]}}(s) = \sqrt{\frac{2}{\pi}} \int_{-\infty}^\infty \left( \int_x^\infty \left( \int_{-\infty}^y n_x(s,y) n_z(T-s,y)
                              \ee^{-\frac{x^2+z^2-y^2}{2}}
                              \,\dd z \right)\dd y \right)\dd x
 \end{align}
 for $s\in(0,T)$.
Note that, by Samorodnitsky and Shen \cite[Theorems 3.1, 3.3 and Example 3.7]{SamShe2},
 $f_{\tau_{R,[0,T]}}$ is continuous on $(0,T)$, $\lim_{s\downarrow 0} f_{\tau_{R,[0,T]}}(s) = \infty$
 and $\lim_{s\uparrow T} f_{\tau_{R,[0,T]}}(s) = \infty$.
Using \eqref{help_taubeta}, for any $0<t_1<t_2<T$, a density function of $\tau_{Z^*,[t_1,t_2]}$ can be chosen as
  \begin{align*}
   f_{\tau_{Z^*,[t_1,t_2]}}(v)
    & = \beta_T'(v) f_{\tau_{R,[0,\beta_T(t_2) - \beta_T(t_1)]}}(\beta_T(v) - \beta_T(t_1)) \\
    & = \frac{\sigma(v)^2}{\phi(v)^2 \int_0^v \frac{\sigma(u)^2}{\phi(u)^2} \,\dd u}
        f_{\tau_{R,[0,\beta_T(t_2) - \beta_T(t_1)]}}(\beta_T(v) - \beta_T(t_1))
 \end{align*}
 for $v\in(t_1,t_2)$, where $f_{\tau_{R,[0,\beta_T(t_2) - \beta_T(t_1)]}}$ is given by \eqref{density_R}.

\section{Examples}\label{sec_examples}

\subsection{Scaled Wiener bridges}\label{subsection_scaled_Wiener_bridge}

Let $T\in(0,\infty)$ be fixed.
For all $\alpha\in\RR$, let us consider the SDE
 \begin{align}\label{alpha_W_bridge}
  \begin{cases}
   \dd Z_t=-\frac{\alpha}{T-t}\,Z_t\,\dd t+\dd B_t,\qquad t\in[0,T),\\
   \phantom{\dd} Z_0=0,
  \end{cases}
 \end{align}
 where $(B_t)_{t\in\RR_+}$ is a standard Wiener process defined
 on a probability space $(\Omega,\cA,\PP)$.
The SDE \eqref{alpha_W_bridge} has a pathwise unique strong solution, namely,
 \begin{align}\label{alpha_W_bridge_intrep}
   Z_t=\int_{0}^t\left(\frac{T-t}{T-s}\right)^\alpha\,\dd B_s,\qquad t\in[0,T),
 \end{align}
 as it can be checked by It\^{o}'s formula.
To our knowledge, these kinds of processes have been first considered in the case of $\alpha>0$
 by Brennan and Schwartz \cite{BreSch}; see also Mansuy \cite{Man}.
Note also that in case of $\alpha=1$ the process $(Z_t)_{t\in[0,T)}$ is nothing else
 but the usual Wiener bridge (from $0$ to $0$ over the time interval $[0,T]$).

It is known that in case of $\alpha>0$, the process $(Z_t)_{t\in[0,T)}$ given by
 \eqref{alpha_W_bridge_intrep} has an almost surely continuous extension $(Z_t)_{t\in[0,T]}$ to the
 time-interval $[0,T]$ such that $Z_T=0$ with probability one, see, e.g.,
 Mansuy \cite[page 1023]{Man} or Barczy and Pap \cite[Lemma 3.1]{BarPap}.
For positive values of $\alpha$, the possibility of such an extension is based on a strong law of
 large numbers for square integrable local martingales.
In case of $\alpha\leq 0$, there does not exist an almost surely continuous extension of the process
 $(Z_t)_{t\in[0,T)}$ to $[0,T]$ which would take some constant value at time $T$ with probability one.
However, for all $\alpha\in\RR$, the Gauss process $(Z_t)_{t\in[0,T)}$ given by \eqref{alpha_W_bridge_intrep}
 is called a scaled Wiener bridge or an $\alpha$-Wiener bridge.
More generally, we call any almost surely continuous (Gauss) process on the time interval
 $[0,T)$ having the same finite-dimensional distributions as $(Z_t)_{t\in[0,T)}$
 a scaled Wiener bridge (an $\alpha$-Wiener bridge).

One can easily calculate
 \begin{align*}
   \cov(Z_s,Z_t)
    = \begin{cases}
         \frac{(T-s)^\alpha (T-t)^\alpha}{1- 2\alpha} \left( T^{1-2\alpha} - (T-s\wedge t)^{1-2\alpha} \right)
              & \text{if \ $\alpha\ne\frac{1}{2}$,}\\[1mm]
          \sqrt{(T-s)(T-t)} \left(\ln(T) - \ln(T-s\wedge t)\right)
              & \text{if \ $\alpha=\frac{1}{2}$,}
      \end{cases}
 \end{align*}
 for $s,t\in[0,T)$, see Barczy and Pap \cite[Lemma 2.1]{BarPap}.

Let $\phi:[0,T)\to\RR$, $\phi(t):=(1-t/T)^\alpha$, $t\in[0,T)$, $\psi:[0,T)\to\RR$, $\psi(t):=0$, $t\in[0,T)$,
 and $\sigma(t):=1$, $t\in[0,T)$.
Then the SDE \eqref{SDE_general} is nothing else but the SDE of an $\alpha$-Wiener bridge,
 see \eqref{alpha_W_bridge}, and
 \begin{align}\label{help_scaled_Wiener1}
  \ln\left( \int_0^t \frac{\sigma(u)^2}{\phi(u)^2}\,\dd u \right)
    = \begin{cases}
        \ln\left(\frac{T^{2\alpha}}{1-2\alpha}\left(T^{1-2\alpha}  - (T-t)^{1-2\alpha} \right) \right)
                 & \text{if \ $\alpha\ne\frac{1}{2}$,}\\[1mm]
         \ln\left(T\ln \left( \frac{T}{T-t}\right) \right)
                 & \text{if \ $\alpha=\frac{1}{2}$,}
      \end{cases}
 \end{align}
 for $t\in(0,T)$.
In case of $\alpha>0$, since $\PP(\lim_{t\uparrow T} Z_t =0)=1$ and $\EE(Z_t)=0$, $t\in[0,T)$,
 we have $\PP(\lim_{t\uparrow T} \widetilde Z_t =0)=1$, so, by Theorem \ref{Thm_General},
 there exists a strictly stationary centered Ornstein-Uhlenbeck process $R=(R_t)_{t\in\RR}$ with
 $\cov(R_s,R_t) = \ee^{-\frac{\vert t-s\vert}{2}}$, $s,t\in\RR$, such that
 \begin{align}\label{help1}
       Z_t = \sqrt{\varr(Z_t)} \,R\left(\ln\left(\frac{T^{2\alpha}}{1-2\alpha}\big(T^{1-2\alpha} - (T-t)^{1-2\alpha}\big)\right)\right),
            \qquad \forall\; t\in[0,T] \quad \text{a.s.}
 \end{align}
 in case $\alpha\ne\frac{1}{2}$, $\alpha>0$, and
 \begin{align}\label{help2}
       Z_t =   \sqrt{\varr(Z_t)} \,
               R\left(\ln \left( T \ln \left(\frac{T}{T-t} \right)\right)\right),
               \qquad \forall\; t\in[0,T] \quad \text{a.s.}
 \end{align}
 in case $\alpha=\frac{1}{2}$, where
 \[
  \sqrt{\varr(Z_t)}
     = \begin{cases}
          (T-t)^\alpha \sqrt{\frac{T^{1-2\alpha} - (T-t)^{1-2\alpha}}{1-2\alpha}}
              & \text{if \ $\alpha\ne\frac{1}{2}$,}\\[2mm]
        \sqrt{(T-t) \ln\left(\frac{T}{T-t}\right)}
              & \text{if \ $\alpha = \frac{1}{2}$.}
      \end{cases}
 \]
Here the right hand sides of \eqref{help1} and \eqref{help2} for $t=0$ and $t=T$ are understood as almost sure limits as $t\downarrow 0$
 and $t\uparrow T$, respectively.

\begin{remark}
As it was mentioned, in case of $\alpha>0$, we have $\PP(\lim_{t\uparrow T} Z_t =0)=1$.
We point out to the fact that Proposition \ref{Prop_General} also yields it.
Indeed, $\EE(Z_t)=0$, $t\in[0,T)$, and Proposition \ref{Prop_General} can be applied with
 \[
     \varepsilon:= \begin{cases}
                    \frac{1}{2} & \text{if \ $0<\alpha\leq \frac{1}{2}$,}\\
                    \frac{1}{2(2\alpha-1)} & \text{if \ $\alpha>\frac{1}{2}$.}
                    \end{cases}
 \]
Indeed, if $0<\alpha<1/2$ and $\varepsilon=1/2$, then
 \begin{align*}
   \phi(t) \left(\int_0^t \frac{\sigma(u)^2}{\phi(u)^2}\,\dd u\right)^{1/2+\varepsilon}
      =  \left( 1 - \frac{t}{T} \right)^\alpha
         \frac{T^{2\alpha}}{1-2\alpha} (T^{1-2\alpha} - (T-t)^{1-2\alpha})
      \to 0
 \end{align*}
 as \ $t\downarrow 0$ \ or \ $t\uparrow T$.
If $\alpha=1/2$ and $\varepsilon=1/2$, then
 \begin{align*}
   \phi(t) \left(\int_0^t \frac{\sigma(u)^2}{\phi(u)^2}\,\dd u\right)^{1/2+\varepsilon}
    = \sqrt{1-\frac{t}{T}} T \ln\left(\frac{T}{T-t}\right)
    \to T\ln(1)=0 \quad \text{as \ $t\downarrow 0$,}
 \end{align*}
 and, by $\cL$'Hospital's rule,
 \begin{align*}
   \lim_{t\uparrow T}\phi(t) \left(\int_0^t \frac{\sigma(u)^2}{\phi(u)^2}\,\dd u\right)^{1/2+\varepsilon}
    = \lim_{t\uparrow T} \frac{T \ln\left(\frac{T}{T-t}\right)}{\left(1-\frac{t}{T}\right)^{-1/2}}
    = \lim_{t\uparrow T} 2\sqrt{T(T-t)}
    = 0.
 \end{align*}
If $\alpha>1/2$ and $\varepsilon=1/(2(2\alpha-1))$, then
 \begin{align*}
  \phi(t) \left(\int_0^t \frac{\sigma(u)^2}{\phi(u)^2}\,\dd u\right)^{1/2+\varepsilon}
    = \left(\frac{T^{2\alpha}}{2\alpha-1}\right)^{\alpha/(2\alpha-1)}
      \frac{1}{T^\alpha}\big(1-T^{1-2\alpha}(T-t)^{2\alpha-1}\big),
 \end{align*}
 which tends to $0$ as $t\downarrow 0$ and to $(2\alpha-1)^{-\alpha/(2\alpha-1)}T^{\alpha/(2\alpha-1)}$ as $t\uparrow T$.
All in all, with the given $\varepsilon$ the function given in \eqref{help_boundedness} is bounded.
\proofend
\end{remark}

\begin{remark}
Note that if $\alpha=1$ and $T=1$, then
 \begin{align*}
    &\sqrt{\var(Z_t)}  = \sqrt{t(1-t)},\qquad t\in[0,1], \\
    &\ln\left(\frac{T^{2\alpha}}{1-2\alpha}\big(T^{1-2\alpha} - (T-t)^{1-2\alpha}\big)\right)
       = \ln\left(\frac{t}{1-t}\right), \qquad t\in(0,1),
 \end{align*}
 and \eqref{help1} gives back the representation of a usual Wiener bridge
 (from $0$ to $0$ on the time interval $[0,1]$) via a space-time scaled stationary Ornstein-Uhlenbeck process,
  recalled in the Introduction.
\proofend
\end{remark}

Next we point out that the previous results presented for the $\alpha$-Wiener bridge given by the SDE \eqref{alpha_W_bridge}
 can be generalized to the so-called general $\alpha$-Wiener bridges introduced by Barczy and Kern \cite[Section 3]{BarKer1}.
Let $T\in(0,\infty)$ be fixed, and for a continuously differentiable function $\alpha: [0,T)\to\RR$, let us consider
 the SDE
 \begin{align}\label{gen_alpha_W_bridge}
  \begin{cases}
   \dd Z_t = -\frac{\alpha(t)}{T-t}Z_t \, \dd t + \dd B_t, \qquad t\in[0,T),\\
   \phantom{\dd} Z_0=0.
  \end{cases}
 \end{align}
This SDE has a pathwise unique strong solution given by
 \begin{align}\label{gen_alpha_W_bridge_intrep}
   Z_t = \int_0^t \exp\left\{ - \int_s^t\frac{\alpha(u)}{T-u}\,\dd u\right\}\dd B_s,
   \qquad t\in[0,T),
 \end{align}
 see, e.g., Barczy and Kern \cite[Proposition 3.1]{BarKer1}.
Further, if $\alpha(T):=\lim_{t\uparrow T}\alpha(t)$ exists and $\alpha(T)>0$, then
 the process $(Z_t)_{t\in[0,T)}$ given by \eqref{gen_alpha_W_bridge_intrep} has an almost surely
 continuous extension $(Z_t)_{t\in[0,T]}$ to the time-interval $[0,T]$ such that $Z_T=0$ with probability one,
 see, Barczy and Kern \cite[Theorem 3.3]{BarPap}.

Let $\phi:[0,T)\to\RR$, $\phi(t):=\exp\left\{ - \int_0^t \frac{\alpha(u)}{T-u}\,\dd u\right\}$, $t\in[0,T)$,
 $\psi:[0,T)\to\RR$, $\psi(t):=0$, $t\in[0,T)$, and $\sigma(t):=1$, $t\in[0,T)$.
Then the SDE \eqref{SDE_general} is nothing else but the SDE of a general $\alpha$-Wiener bridge given
 in \eqref{gen_alpha_W_bridge}, and, using Theorem \ref{Thm_General} and $\EE(Z_t) = 0$, $t\in[0,T)$,
 there exists a strictly stationary centered Ornstein-Uhlenbeck process $R=(R_t)_{t\in\RR}$
 with $\cov(R_s,R_t) = \ee^{-\frac{\vert t-s\vert}{2}}$, $s,t\in\RR$, such that
 \[
       Z_t = \phi(t)\sqrt{\int_0^t \frac{1}{\phi(u)^2} \,\dd u}
             \,R\left(\ln\left(\int_0^t \frac{1}{\phi(u)^2} \,\dd u \right)\right),
            \qquad \forall\; t\in[0,T), \quad \text{a.s.,}
 \]
 where the right hand side of the above equality at $t=0$ is understood as an almost sure limit as $t\downarrow 0$.
Further, if $\alpha(T):=\lim_{t\uparrow T}\alpha(t)$ exists and $\alpha(T)>0$,
 then $\PP(\lim_{t\uparrow T} Z_t =0)=1$ and $\EE(Z_t)=0$, $t\in[0,T)$, yield that
 $\PP(\lim_{t\uparrow T} \widetilde Z_t =0)=1$, so, by Theorem \ref{Thm_General},
  there exists a strictly stationary centered
 Ornstein-Uhlenbeck process $R=(R_t)_{t\in\RR}$ with
 $\cov(R_s,R_t) = \ee^{-\frac{\vert t-s\vert}{2}}$, $s,t\in\RR$, such that
 \[
       Z_t = \phi(t)\sqrt{\int_0^t \frac{1}{\phi(u)^2} \,\dd u}
             \,R\left(\ln\left(\int_0^t \frac{1}{\phi(u)^2} \,\dd u \right)\right),
            \qquad \forall\; t\in[0,T], \quad \text{a.s.,}
 \]
 where the right hand side of the above equation for $t=0$ and for $t=T$ is understood
 as an almost sure limit as $t\downarrow 0$ and $t\uparrow T$, respectively.

\begin{remark}
As it was mentioned, if $\alpha(T)=\lim_{t\uparrow T}\alpha(t)$ exists and $\alpha(T)>0$, then
 we have $\PP(\lim_{t\uparrow T} Z_t =0)=1$.
We point out to the fact that Proposition \ref{Prop_General} also yields it.
Indeed, $\EE(Z_t)=0$, $t\in[0,T)$, and Proposition \ref{Prop_General} can be applied with
 \[
  \varepsilon
   := \begin{cases}
        \frac{1+2(\delta_1 - \delta_2)}{2(2\delta_2 -1)} & \text{if \ $\alpha(T)\geq \frac{1}{2}$,}\\[1mm]
         \frac{1}{2} & \text{if \ $0<\alpha(T)<\frac{1}{2}$,}
      \end{cases}
 \]
 where $\delta_1$ and $\delta_2$ are chosen such that $0<\delta_1<\alpha(T)<\delta_2<\delta_1+1/2$.
We need to check that the function
 \begin{align*}
   \phi(t)&\left(\int_0^t \frac{\sigma(u)^2}{\phi(u)^2}\,\dd u\right)^{\frac{1}{2}+\varepsilon}\\
          & = \exp\left\{ - \int_0^t \frac{\alpha(u)}{T-u}\,\dd u\right\}
              \left( \int_0^t \exp\left\{ 2\int_0^u \frac{\alpha(v)}{T-v}\,\dd v\right\} \,\dd u \right)^{\frac{1}{2}+\varepsilon},
             \qquad t\in[0,T),
 \end{align*}
 is bounded.
Let us choose a $t_0\in(0,T)$ such that $\delta_1\leq \alpha(t)\leq \delta_2$ for all $t\in[t_0,T]$.

First, we consider the case $\alpha(T)\geq 1/2$.
Since
 \begin{align*}
  \int_0^t \exp\left\{ 2\int_0^u \frac{\alpha(v)}{T-v}\,\dd v\right\} \,\dd u
    = C_1 + C_2 \int_{t_0}^t \exp\left\{ 2\int_{t_0}^u \frac{\alpha(v)}{T-v}\,\dd v\right\} \,\dd u,
    \qquad t\in[t_0,T),
 \end{align*}
 where
 \begin{align*}
   C_1:= \int_0^{t_0} \exp\left\{ 2\int_0^u \frac{\alpha(v)}{T-v}\,\dd v\right\} \,\dd u
   \qquad \text{and} \qquad
   C_2:= \exp\left\{ 2\int_0^{t_0} \frac{\alpha(v)}{T-v}\,\dd v\right\},
 \end{align*}
 we have for all $t\in[t_0,T)$
 \begin{align*}
   \exp&\left\{ - \int_0^t \frac{\alpha(u)}{T-u}\,\dd u\right\}
        \left( \int_0^t \exp\left\{ 2\int_0^u \frac{\alpha(v)}{T-v}\,\dd v\right\} \,\dd u \right)^{\frac{1}{2}+\varepsilon}\\
   & \leq C_3 \exp\left\{ -\delta_1\int_{t_0}^t \frac{1}{T-v}\,\dd v\right\}
           \left(C_1 + C_2 \int_{t_0}^t \exp\left\{ 2\delta_2\int_{t_0}^u \frac{1}{T-v}\,\dd v\right\}
            \,\dd u \right)^{\frac{1}{2}+\varepsilon}\\
   & = C_3 \left(\frac{T-t}{T-t_0}\right)^{\delta_1}
           \left(C_1 + C_2\frac{(T-t_0)^{2\delta_2}}{2\delta_2-1}
                        \big((T-t)^{1-2\delta_2} - (T-t_0)^{1-2\delta_2}\big)
                 \right)^{\frac{1}{2}+\varepsilon}\\
   &\leq C_3 \left(\frac{T-t}{T-t_0}\right)^{\delta_1}
            \left(C_1 + C_2\frac{(T-t_0)^{2\delta_2}}{2\delta_2-1}
                        (T-t)^{1-2\delta_2}
                 \right)^{\frac{1}{2}+\varepsilon}\\
   & = C_3 \left( C_1\left(\frac{T-t}{T-t_0}\right)^{\frac{2\delta_1}{2\varepsilon+1}}
                      + \frac{C_2}{2\delta_2-1} (T-t_0)^{2\delta_2 - \frac{2\delta_1}{2\varepsilon+1}}
                         (T-t)^{1-2\delta_2 + \frac{2\delta_1}{2\varepsilon+1}}
            \right)^{\frac{1}{2}+\varepsilon},
 \end{align*}
 where
 \[
   C_3:=\exp\left\{-\int_0^{t_0} \frac{\alpha(u)}{T-u}\,\dd u\right\}.
 \]
Here, using that $2\delta_2 - 1>0$, $\delta_2-\delta_1<1/2$, and the explicit form of $\varepsilon$,
 one can easily verify that
 $1-2\delta_2 + \frac{2\delta_1}{2\varepsilon+1}>0$, yielding that the function
 \begin{align*}
   \phi(t)\left(\int_0^t \frac{\sigma(u)^2}{\phi(u)^2}\,\dd u\right)^{\frac{1}{2}+\varepsilon},
   \quad t\in[0,T),
 \end{align*}
 is bounded in case of $\alpha(T)\geq 1/2$.

Next, we consider the case $0<\alpha(T)<1/2$.
Additionally to $0<\delta_1<\alpha(T)<\delta_2<\delta_1+1/2$, we can also assume that $\delta_2<1/2$.
By the calculations for the case $\alpha(T)\geq 1/2$, we get
 \begin{align*}
  \int_0^t \exp\left\{ 2\int_0^u \frac{\alpha(v)}{T-v}\,\dd v\right\} \,\dd u
   & \leq C_1 + C_2 \frac{(T-t_0)^{2\delta_2}}{2\delta_2-1}
             \big( (T-t)^{1-2\delta_2} - (T-t_0)^{1-2\delta_2} \big) \\
   & \to C_1 + C_2\frac{T-t_0}{1-2\delta_2}
    \qquad \text{as \ $t\uparrow T$,}
 \end{align*}
 and
 \[
   \exp\left\{-\int_0^t \frac{\alpha(v)}{T-v}\,\dd v\right\}
       \leq C_3 \left(\frac{T-t}{T-t_0}\right)^{\delta_1}
       \to 0
       \qquad \text{as \ $t\uparrow T$,}
 \]
 where we used $1-2\delta_2>0$ and $\delta_1>0$.
This yields that the function
 \begin{align*}
   \phi(t)\left(\int_0^t \frac{\sigma(u)^2}{\phi(u)^2}\,\dd u\right)^{\frac{1}{2}+\varepsilon},
   \quad t\in[0,T),
 \end{align*}
 is bounded also in case of $0<\alpha(T)<1/2$.
\end{remark}

\subsection{Ornstein-Uhlenbeck type bridges}\label{subsection_OU}

First, we recall the notion and properties of Ornstein-Uhlenbeck type bridges to
the extent needed.
For a more detailed discussion and for the proofs of the results,
 see for example Barczy and Kern \cite{BarKer2} (where one can also find extensions to multidimensional bridges).

Let us consider an Ornstein-Uhlenbeck type process  $(X_t)_{t\in\RR_+}$ given by the SDE
 \begin{align}\label{gen_OU_egyenlet}
     \dd X_t=q(t)\,X_t\,\dd t + \sigma(t)\,\dd B_t,\qquad t\in\RR_+,
 \end{align}
 with an initial condition $X_0$ having a Gauss distribution independent of $B$,
 where $q:\RR_+\to\RR$ and $\sigma:\RR_+\to\RR$  are continuous functions
 and $(B_t)_{t\in\RR_+}$ is a standard Wiener process defined
 on a probability space $(\Omega,\cA,\PP)$.
Note that the SDE \eqref{gen_OU_egyenlet} coincides with the SDE \eqref{SDE_general} with $T=\infty$,
 $\phi(t) := \exp\left\{\int_0^t q(u)\,\dd u \right\}$, $t\in\RR_+$, and $\psi(t):=0$, $t\geq 0$.
By It\^{o}'s formula, there exists a pathwise unique strong solution of the SDE \eqref{gen_OU_egyenlet},
namely, for $t\in\RR_+$,
\begin{align}\label{gen_OU_egyenlet_megoldas}
  X_t=\ee^{\bar q(t)}\left(X_0 + \int_0^t \ee^{-\bar q(s)} \sigma(s) \,\dd B_s\right)\quad\text{ with }
       \quad \bar q(t):=\int_{0}^tq(u)\,\dd u.
 \end{align}

Let us introduce the following notations and assumptions.
Let
 \[
   \gamma(s,t):=\int_{s}^t\ee^{2(\bar q(t)-\bar q(u))}\sigma^2(u)\,\dd u<\infty,
   \qquad 0\leq s\leq t.
 \]
In what follows we will make the general assumption that
 \begin{align}\label{sigma_assumption}
       \sigma(t)\ne0  \qquad \text{for all $t\in\RR_+$},
 \end{align}
 which guarantees that $\gamma(s,t)$ is positive for all $0\leq s<t$.
Further, for all $a,b\in\RR$ and $0\leq s\leq t<T<\infty$, let
\begin{equation}\label{gen_OU_exp}
 n_{a,b}(s,t)
  :=\frac{\gamma(s,t)}{\gamma(s,T)}\,\ee^{\bar q(T)-\bar q(t)} b
     +\frac{\gamma(t,T)}{\gamma(s,T)}\ee^{\bar q(t)-\bar q(s)} a,
 \end{equation}
 and
 \begin{equation}\label{gen_OU_var}
  \sigma(s,t):=\frac{\gamma(s,t)\,\gamma(t,T)}{\gamma(s,T)}.
 \end{equation}

In Barczy and Kern \cite{BarKer2}, for fixed $T\in(0,\infty)$ and $a,b\in\RR$ we constructed
a Markov process $(Z_t)_{t\in[0,T]}$ with initial distribution
$\PP(Z_0=a)=1$ and with transition densities
 \begin{equation}\label{gen_bridge_densities}
    p_{s,t}^Z(x,y)=\frac{p_{s,t}^X(x,y)\,p_{t,T}^X(y,b)}{p_{s,T}^X(x,b)},
      \quad x,y\in\RR,\quad 0\leq s<t<T,
 \end{equation}
 such that $Z_{t}\to b=Z_T$ almost surely and also in $L^2$ as $t\uparrow T$,
 where $p_{s,t}^X$ denotes the transition densities of $X$.
The process $(Z_t)_{t\in[0,T]}$ is called a bridge of Ornstein-Uhlenbeck type from
$a$ to $b$ over $[0,T]$ derived from $X$, see also Definition \ref{DEFINITION_bridge}.
The construction is based on Theorem 3.1 in Barczy and Kern \cite{BarKer2}, which we recall now
 for completeness and for our later purposes.
For the proofs, see Barczy and Kern \cite{BarKer2}.

\begin{theorem}\label{THEOREM2}
 Let us suppose that condition \eqref{sigma_assumption} holds.
For fixed $a,b\in\RR$ and $T\in(0,\infty)$, let the process $(Z_t)_{t\in[0,T)}$ be given by
 \begin{align}\label{gen_OU_integral}
   \begin{split}
    Z_t=  n_{a,b}(0,t)
           +\int_{0}^t\frac{\gamma(t,T)}{\gamma(s,T)}\ee^{\bar q(t)-\bar q(s)} \sigma(s)\,\dd B_{s},
          \qquad t\in[0,T).
    \end{split}
 \end{align}
Then for any $t\in[0,T)$ the distribution of $Z_t$ is Gauss
 with mean $n_{a,b}(0,t)$ and with variance $\sigma(0,t).$ Especially, $Z_t\to b$ almost
 surely (and hence in probability) and in $L^2$ as $t\uparrow T$.
Hence the process $(Z_t)_{t\in[0,T)}$ can be extended to an almost surely (and hence stochastically) and
 $L^2$-continuous process $(Z_t)_{t\in[0,T]}$ with $Z_0=a $ and $Z_T=b$.
Moreover, $(Z_t)_{t\in[0,T]}$ is a Gauss-Markov process
 and for any $x\in\RR$ and $0\leq s<t<T$ the transition density
 $\RR\ni y\mapsto p_{s,t}^Z(x,y)$ of $Z_t$ given $Z_s=x$ is given by
 \begin{align*}
   p_{s,t}^Z(x,y)
     = \frac{1}{\sqrt{2\pi\sigma(s,t)}}
    \exp\left\{-\frac{\left(y-n_{x,b}(s,t)\right)^2}
                      {2\sigma(s,t)}\right\},
        \qquad y\in\RR.
 \end{align*}
\end{theorem}

\begin{defi}\label{DEFINITION_bridge}
Let $(X_t)_{t\in\RR_+}$ be the process given by the SDE \eqref{gen_OU_egyenlet} with an initial
 Gauss random variable $X_0$ independent of $(B_t)_{t\in\RR_+}$
 and let us assume that condition \eqref{sigma_assumption} holds.
For fixed $a,b\in\RR$ and $T\in(0,\infty)$, the process $(Z_t)_{t\in[0,T]}$ defined in Theorem
 \ref{THEOREM2} is called a bridge of Ornstein-Uhlenbeck type from $a$ to $b$ over $[0,T]$ derived
 from $X$.
More generally, we call any almost surely continuous (Gauss) process on the time-interval
 $[0,T]$ having the same finite-dimensional distributions as $(Z_t)_{t\in[0,T]}$ a bridge
 of Ornstein-Uhlenbeck type from $a$ to $b$ over $[0,T]$ derived from $X$.
\end{defi}

One can also derive a SDE which is satisfied by the Ornstein-Uhlenbeck type bridge,
 see for example Theorem 3.3 in Barczy and Kern \cite{BarKer2}.

\begin{theorem}\label{LEMMA4}
Let us suppose that condition \eqref{sigma_assumption} holds.
The process $(Z_t)_{t\in[0,T)}$ defined by \eqref{gen_OU_integral} is a pathwise unique
 strong solution of the linear SDE
 \begin{align}\label{gen_OU_hid1_egyenlet}
  \begin{split}
   \dd Z_t= \left[\left(q(t)-\frac{\ee^{2(\bar q(T)-\bar q(t))}}{\gamma(t,T)}\, \sigma^2(t) \right)\,Z_{t}
            + \frac{\ee^{\bar q(T)-\bar q(t)}}{\gamma(t,T)}\, \sigma^2 (t) b\right]\dd t
            + \sigma(t)\,\dd B_t
 \end{split}\end{align}
for $t\in[0,T)$ and with initial condition $Z_0=a$.
\end{theorem}

By Lemma 2.7 in Barczy and Kern \cite{BarKer2}, one can easily calculate
 \begin{align*}
   \cov(Z_s,Z_t)
     = \ee^{\bar q(t) - \bar q(s)} \frac{\gamma(0,s) \gamma(t,T) }{\gamma(0,T)},
     \qquad 0\leq s\leq t<T.
 \end{align*}

Let us define the mean centered Ornstein-Uhlenbeck type bridge
 \[
    \widetilde Z_t
      := Z_t - \EE(Z_t)
       = Z_t - n_{a,b}(0,t), \qquad t\in[0,T].
 \]
Note that $\PP(\widetilde Z_0 = 0) = \PP(\widetilde Z_T = 0)=1$.

With the notations of Section \ref{SDE_general}, let $\xi:=a$, $\phi:[0,T)\to(0,\infty)$ and $\psi:[0,T)\to\RR$ be defined by
 \begin{align*}
   &\phi(t):=\frac{\gamma(t,T)\ee^{\bar q(t)}}{\gamma(0,T)}, \qquad t\in[0,T),\\
   &\psi(t):=\frac{\ee^{\bar q(T) - \bar q(t)}}{\gamma(t,T)}\sigma^2(t) b
            = \frac{\ee^{\bar q(T)}}{\gamma(0,T)\phi(t)}\sigma^2(t)b,\qquad t\in[0,T).
 \end{align*}
Since
 \[
   \frac{\phi'(t)}{\phi(t)}
     = q(t) - \sigma^2(t)\frac{\ee^{2(\bar q(T) - \bar q(t))}}{\gamma(t,T)},
     \qquad t\in[0,T),
 \]
 with our special choices of $\xi$, $\phi$ and $\psi$, the SDE \eqref{SDE_general} is nothing else but the SDE of an Ornstein-Uhlenbeck bridge from $a$ to $b$ over time interval $[0,T]$,
 see \eqref{gen_OU_hid1_egyenlet}.
Further, using part (b) of Lemma A.3 in Barczy and Kern \cite{BarKer2}, one can check that
 \begin{align*}
   \int_0^t \frac{\sigma(u)^2}{\phi(u)^2}\,\dd u
    = \frac{\ee^{-2\bar q(t)}\gamma(0,t)\gamma(0,T)}{\gamma(t,T)},
      \qquad t\in[0,T).
  \end{align*}
Since $Z_t\to b$ almost surely and in $L^2$ as $t\uparrow T$ (especially, $\EE(Z_t)\to b$ as $t\uparrow T$), we have
 $\PP(\widetilde Z_t \to 0 \ \text{as $t\uparrow T$})=1$, so, by Theorem \ref{Thm_General},
 there exists a strictly stationary centered Ornstein-Uhlenbeck process $R=(R_t)_{t\in\RR}$ with
 $\cov(R_s,R_t) = \ee^{-\frac{\vert t-s\vert}{2}}$, $s,t\in\RR$, such that
 \begin{align}\label{help3}
       \widetilde Z_t = \sqrt{\frac{\gamma(0,t) \gamma(t,T) }{\gamma(0,T)}}\,
             R\left(\ln\left(\frac{\ee^{-2\bar q(t)}\gamma(0,t)\gamma(0,T)}{\gamma(t,T)}\right)
              \right),
             \qquad \forall\; t\in[0,T], \quad \text{a.s.,}
 \end{align}
  where the right hand side of the above equation for $t=0$ and for $t=T$ is understood
 as an almost sure limit as $t\downarrow 0$ and $t\uparrow T$, respectively.

\begin{remark}
As it was mentioned, $\PP(\lim_{t\uparrow T} \widetilde Z_t =0)=1$ holds.
We point out to the fact that Proposition \ref{Prop_General} also yields it.
Indeed, Proposition \ref{Prop_General} can be applied with $\varepsilon:=1/2$,
 since
 \begin{align*}
   \phi(t) \left(\int_0^t \frac{\sigma(u)^2}{\phi(u)^2}\,\dd u\right)^{1/2+\varepsilon}
     & = \frac{\gamma(t,T)\ee^{\bar q(t)}}{\gamma(0,T)}
        \frac{\ee^{-2\bar q(t)}\gamma(0,t)\gamma(0,T)}{\gamma(t,T)}
       = \ee^{-\bar q(t)} \gamma(0,t)\\
     & = \ee^{\bar q(t)} \int_0^t \ee^{-2\bar q(u)}\sigma(u)^2\,\dd u,
      \qquad t\in[0,T),
 \end{align*}
 which is a bounded function, since the functions $q$ and $\sigma$ are continuous on $\RR_+$.
\proofend
\end{remark}

Next, we formulate the above presented results in the case of usual Ornstein-Uhlenbeck bridges.

\begin{remark}\label{REMARK3}
In case of $q(t)=q\ne0,$ $t\in\RR_+$, and $\sigma(t)=\sigma\ne0,$ $t\in\RR_+$,
 the bridge of Ornstein-Uhlenbeck type $(Z_t)_{t\in[0,T]}$ from $a$ to $b$ over
 $[0,T]$ defined in \eqref{gen_OU_integral} has the form
 \begin{align}\label{OU_hid1_intrep}
   Z_t=a\,\frac{\sinh(q(T-t))}{\sinh(qT)}
         + b\,\frac{\sinh(qt)}{\sinh(qT)}
         + \sigma  \int_0^t\frac{\sinh(q(T-t))}{\sinh(q(T-s))}\,\dd B_s
 \end{align}
 for $t\in[0,T)$ and $Z_T=b$, see, Remark 3.8 in Barczy and Kern \cite{BarKer2}.
In fact, the process $(Z_t)_{t\in[0,T)}$ is the pathwise unique strong solution of the SDE
 \[
  \dd Z_t = q\left(-\coth(q(T-t))Z_t + \frac{b}{\sinh(q(T-t))}\right)\dd t
        + \sigma\,\dd B_t,\qquad t\in[0,T),
 \]
 with an initial condition $Z_0=a$, see, Remark 3.9 in Barczy and Kern \cite{BarKer2}.
Then, by \eqref{help3}, there exists a strictly stationary centered Ornstein-Uhlenbeck process $R=(R_t)_{t\in\RR}$ with
  $\cov(R_s,R_t) = \ee^{-\frac{\vert t-s\vert}{2}}$, $s,t\in\RR$, such that
 \[
       \widetilde Z_t = \sqrt{\frac{\sigma^2\sinh(qt)\sinh(q(T-t))}{q\sinh(qT)}}
             R\left(\ln\left(\frac{\sigma^2\sinh(qt)\sinh(qT)}{q\sinh(q(T-t))}\right)
                    \right),
             \;\; \forall\; t\in[0,T], \;\; \text{a.s.,}
 \]
 where the right hand side of the above equation for $t=0$ and for $t=T$ is understood
 as an almost sure limit as $t\downarrow 0$ and $t\uparrow T$, respectively,
 since
 \[
   \gamma(s,t) = \frac{\sigma^2}{q}\ee^{q(t-s)}\sinh(q(t-s)),\qquad 0\leq s\leq t,
 \]
 and, by Barczy and Kern \cite[formula (1.7)]{BarKer2},
 \[
   \var(Z_t) = \phi(t)^2 \int_0^t \frac{\sigma(u)^2}{\phi(u)^2}\,\dd u
            = \frac{\sigma^2\sinh(qt)\sinh(q(T-t))}{q\sinh(qT)},
            \qquad t\in[0,T).
 \]
\proofend
\end{remark}

\subsection{$F$-Wiener bridges}
Let $f:\RR_+\to\RR_+$ be a probability density function on $\RR_+$ and
 let us consider the corresponding cumulative distribution function $F:\RR_+\to[0,1]$, $F(t):=\int_0^t f(s)\,\dd s$,
 $t\in\RR_+$.
Further, let
 \[
    T:=\inf\{t\in\RR_+ : F(t)=1\}\in(0,\infty]
 \]
 with the convention $\inf\emptyset:=\infty$.
Let us assume that $f$ is continuous on $[0,T)$, and that there exists a $\delta\in(0,T)$
 such that $f(t)\ne 0$ for all $t\in(0,\delta)$.
We consider the SDE
 \begin{align}\label{SDE_Khmaladze}
   \dd Z_t = -\frac{f(t)}{1- F(t)}Z_t\,\dd t + \sqrt{f(t)}\,\dd B_t, \qquad t\in[0,T),
 \end{align}
 with an initial value $Z_0=0$, where $(B_t)_{t\in\RR_+}$ is a standard Wiener process.
By It\^{o}'s formula,
 \[
    Z_t = \int_0^t \frac{1-F(t)}{1-F(s)}\sqrt{f(s)}\,\dd B_s, \qquad t\in[0,T),
 \]
 is a pathwise unique strong solution of the SDE \eqref{SDE_Khmaladze},
 and $(Z_t)_{t\in[0,T)}$ is a centered Gauss process with covariance function
 \begin{align*}
   \cov(Z_s,Z_t)
    & = (1-F(t))(1-F(s))\int_0^{s\wedge t}\frac{f(u)}{(1-F(u))^2}\,\dd u \\
    & = (1-F(t))(1-F(s)) \frac{F(s\wedge t)}{1-F(s\wedge t)}
      = F(s\wedge t) - F(s)F(t)
 \end{align*}
 for $s,t\in[0,T)$.
Note that $1-F(t)$, $t\in\RR_+$, is nothing else but the survival function, and
 $\frac{f(t)}{1-F(t)}$, $t\in\RR_+$, is the hazard rate (mean reversion rate) corresponding
 to the distribution function $F$.
Let $\phi:[0,T)\to(0,\infty)$, $\phi(t):=1-F(t)$, $t\in[0,T)$, $\psi:[0,T)\to\RR$, $\psi(t):=0$, $t\in[0,T)$,
 and $\sigma(t):=\sqrt{f(t)}$, $t\in[0,T)$.
Then the SDE \eqref{SDE_general} is nothing else but the SDE \eqref{SDE_Khmaladze}, and
 \begin{align*}
  \int_0^t\frac{\sigma(u)^2}{\phi(u)^2}\,\dd u
    = \int_0^t\frac{f(u)}{(1-F(u))^2}\,\dd u
    = \int_0^t\frac{F'(u)}{(1-F(u))^2} \,\dd u
    = \frac{F(t)}{1-F(t)},
    \qquad t\in[0,T).
 \end{align*}
Using Theorem \ref{Thm_General} and $\EE(Z_t) = 0$, $t\in[0,T)$, there exists a strictly stationary centered Ornstein-Uhlenbeck process
 $R=(R_t)_{t\in\RR}$ with $\cov(R_s,R_t) = \ee^{-\frac{\vert t-s\vert}{2}}$, $s,t\in\RR$,
 such that
 \[
   Z_t = \sqrt{F(t)(1-F(t))}\, R\left(\ln\left(\frac{F(t)}{1-F(t)}\right)\right), \qquad \forall\, t\in[0,T), \quad \text{a.s.,}
 \]
  where the right hand side of the above equation for $t=0$  is understood as an almost sure limit as $t\downarrow 0$.
Further, note that with $\varepsilon:=1/2$ we have
 \begin{align*}
   \phi(t) \left(\int_0^t \frac{\sigma(u)^2}{\phi(u)^2}\,\dd u\right)^{1/2+\varepsilon}
      = (1-F(t)) \frac{F(t)}{1-F(t)} = F(t), \qquad t\in[0,T),
 \end{align*}
 which is a bounded function.
Since $F$ is a continuous distribution function, $\lim_{t\uparrow T} F(t) = 1$, and  $\EE(Z_t)=0$, $t\in[0,T)$,
 by Proposition \ref{Prop_General}, we have $\PP(\lim_{t\uparrow T} Z_t = 0)=1$ and
 \[
   Z_t = \sqrt{F(t)(1-F(t))}\, R\left(\ln\left(\frac{F(t)}{1-F(t)}\right)\right), \qquad \forall\, t\in[0,T], \quad \text{a.s.,}
 \]
 where the right hand side of the above equation for $t=0$ and for $t=T$ is understood
 as an almost sure limit as $t\downarrow 0$ and $t\uparrow T$, respectively.
Then we can say that $Z$ is a bridge over $[0,T]$ in the sense that its starting and ending points are zero
 (more precisely, $Z_0=0$ and $\PP(\lim_{t\uparrow T} Z_t = 0)=1$), and we can call $Z$ as an $F$-Wiener bridge
 corresponding to the distribution function $F$.
For more information on $F$-Wiener bridges (also called $\PP$-Wiener bridges), see
 Shorack and Wellner \cite[page 838]{ShoWel}, van der Vaart \cite[page 266]{Vaa}
 or Khmaladze \cite[equation (4)]{Khm}.

To give an example, let us consider the cumulative distribution function $F:\RR_+\to[0,1]$ defined by
 \[
    F(t):=\begin{cases}
            1- \left(1-\frac{t}{T}\right)^\alpha & \text{if \ $t\in[0,T)$,}\\
            1 & \text{if \ $t\geq T$,}
           \end{cases}
 \]
 where $\alpha\in(0,\infty)$.
Then $f(t) = \frac{\alpha}{T}\left(1-\frac{t}{T}\right)^{\alpha-1}$ for $t\in[0,T)$, and $f(t)=0$
 for $t\in\RR_+\setminus [0,T)$, $\inf\{t\in\RR_+ : F(t)=1\} = T$ and the SDE \eqref{SDE_Khmaladze}
 of the $F$-Wiener bridge takes the form
 \begin{align*}
   \dd Z_t = -\frac{\alpha}{T-t}Z_t\,\dd t + \sqrt{f(t)}\,\dd B_t, \qquad t\in[0,T),
 \end{align*}
 with an initial value $Z_0=0$.
Note that the drift coefficients of this SDE and of the SDE \eqref{alpha_W_bridge} of an $\alpha$-Wiener bridge
 are the same.
However, the diffusion coefficients are different.

\subsection{Weighted Wiener processes and weighted Wiener bridges}

Let $(B_t)_{t\in\RR_+}$ be a standard Wiener process, and $(B^\circ_t)_{t\in[0,1]}$ be a Wiener bridge
 from $0$ to $0$ over $[0,1]$.
Let $w:\RR_+\to(0,\infty)$ be a continuously differentiable (weight) function such that $w(0)=1$
 (e.g., $w(t)=(1+t)^\alpha$, $t\in\RR_+$, with some $\alpha\in[1,\infty)$).
Let us define
 \[
    Z_t:=w(t)B_t,\qquad t\in\RR_+, \qquad \text{and}\qquad  Z^\circ_t:=w(t)B^\circ_t,\qquad t\in[0,1].
 \]
The process $(Z_t)_{t\in\RR_+}$ can be called a weighted Wiener process, and the process $(Z^\circ_t)_{t\in[0,1]}$ a weighted
 Wiener bridge.
We note that Deheuvels and Martynov \cite{DehMar} considered weighted Wiener processes and weighted Wiener bridges
 with a weight function $t^\alpha$ for some $\alpha\in(0,\infty)$ (however, this weight function is not in our setup,
 since the condition $w(0)=1$ does not hold).
By \eqref{OU_transform} and Subsection \ref{subsection_scaled_Wiener_bridge},
 there exists a strictly stationary centered Ornstein-Uhlenbeck process $R=(R_t)_{t\in\RR}$ with
 $\cov(R_s,R_t) = \ee^{-\frac{\vert t-s\vert}{2}}$, $s,t\in\RR$, such that
 \begin{align}\label{weighted_wiener_rep}
   Z_t = w(t)\sqrt{t} \,R_{\ln(t)}, \qquad \forall\, t\in\RR_+ \quad \text{a.s.,}
 \end{align}
 and
 \begin{align}\label{weighted_wiener_bridge_rep}
    Z^\circ_t = w(t) \sqrt{t(1-t)} \,R\left(\ln\left(\frac{t}{1-t}\right)\right), \qquad \forall\, t\in[0,1] \quad \text{a.s..}
 \end{align}
We point out that weighted Wiener processes and weighted Wiener bridges fit into our general framework
 (see Section \ref{section_gen_fram}), so one can directly apply
 Theorem \ref{Thm_General} and get the representations \eqref{weighted_wiener_rep} and \eqref{weighted_wiener_bridge_rep},
 detailed as follows.
Namely, by It\^{o}'s formula,
  \begin{align}\label{SDE_weighted_Wiener_process}
  & \dd Z_t = \frac{w'(t)}{w(t)} Z_t\,\dd t + w(t)\,\dd B_t,\qquad t\in\RR_+,\\  \label{SDE_weighted_Wiener_bridge}
  & \dd Z^\circ_t = \left(\frac{w'(t)}{w(t)} - \frac{1}{1-t}\right)Z^\circ_t\,\dd t + w(t)\,\dd B_t,\qquad t\in[0,1).
 \end{align}
The SDEs \eqref{SDE_weighted_Wiener_process} and \eqref{SDE_weighted_Wiener_bridge}
 have the form \eqref{SDE_general} by choosing $T:=\infty$, $\phi(t):=w(t)$, $t\in\RR_+$, $\psi(t):=0$,
 $t\in\RR_+$, $\sigma(t):=w(t)$, $t\in\RR_+$, and $T:=1$, $\phi(t):=w(t)(1-t)$, $t\in[0,1)$, $\psi(t):=0$,
 $t\in[0,1)$, $\sigma(t):=w(t)$, $t\in[0,1)$, respectively.
Concerning the time scalings, an easy calculation shows that
 for the SDE \eqref{SDE_weighted_Wiener_process}, we have
 \[
     \int_0^t\frac{\sigma(u)^2}{\phi(u)^2}\,\dd u = t, \qquad t\in\RR_+,
 \]
 and for the SDE \eqref{SDE_weighted_Wiener_bridge},
 \[
     \int_0^t\frac{\sigma(u)^2}{\phi(u)^2}\,\dd u
       = \frac{t}{1-t}, \qquad t\in[0,1),
 \]
 as desired.

\section{Counterexamples}\label{section_counterexamples}

In this section we give counterexamples for bridge processes that cannot be represented
 as a space-time scaled stationary Ornstein-Uhlenbeck process.

\subsection{Zero area Wiener bridge}\label{Zero area Wiener bridge}
Let $(B^\circ_t)_{t\in[0,1]}$ be a Wiener bridge from $0$ to $0$ over $[0,1]$, and let us consider the process
 \[
   B^\circ_t - 6t(1-t)\int_0^1 B^\circ_u\,\dd u, \qquad t\in[0,1],
 \]
 introduced by Deheuvels \cite{Deh}.
According to page 1191 in Deheuvels \cite{Deh}, this process coincides in law with
 a zero area Wiener bridge $(M_t)_{t\in[0,1]}$,
 which is defined by conditioning a standard Wiener process $(B_t)_{t\in[0,1]}$ such that $B_1=0$ and $\int_0^1 B_u\,\dd u = 0$
 (for a precise definition, see G\"orgens \cite[Section 1.1]{Goe}).
The zero area Wiener bridge $(M_t)_{t\in[0,1]}$ is a Gauss process and it has a covariance function
 \[
     \cov(M_s,M_t) = s\wedge t -st -3st(1-s)(1-t),\qquad s,t\in[0,1],
 \]
 see, Deheuvels \cite[Lemma 2.1 with $K=1$]{Deh}.

We check that one cannot find a monotone function $\tau:[0,1]\to \RR$ such that
 \begin{align}\label{help_Goergens_2}
   \cov(M_s,M_t) = \sqrt{\var(M_s)}\sqrt{\var(M_t)}\cov(R_{\tau(s)},R_{\tau(t)}),
    \qquad s,t\in[0,1],
 \end{align}
 where $R=(R_t)_{t\in\RR}$ is a strictly stationary centered Ornstein-Uhlenbeck process with
 $\cov(R_s,R_t) = \ee^{-\frac{\vert t-s\vert}{2}}$, $s,t\in\RR$.
On the contrary, let us suppose that there exists such a function $\tau$.
Then, due to the covariance structure of $R$,
 the covariance function $\cov(M_s,M_t)$, $s,t\in[0,1]$, would be written as a product of a function only of $s$ and
 a function only of $t$, i.e., there would exist some functions $f:[0,1]\to\RR$ and $g:[0,1]\to\RR$
 such that $\cov(M_s,M_t) = f(s)g(t)$, $s,t\in[0,1]$.
Then for all $0\leq s\leq t\leq 1$, we have
 \begin{align*}
  \cov(M_s,M_t)
    = s-st-3st(1-s)(1-t)
    = s(1-t)(1-3t(1-s))
    = f(s)g(t),
 \end{align*}
 which yields that
 \begin{align}\label{help_Goergens}
   1 - 3t(1-s) = \frac{f(s)}{s}\frac{g(t)}{1-t}
               =: \widetilde f(s) \widetilde g(t),
             \qquad 0<s\leq t<1.
 \end{align}
By substituting $s:=1/2$ into \eqref{help_Goergens}, we have
 \[
   1-\frac{3}{2}t = \widetilde f(1/2)\widetilde g(t), \qquad t\in[1/2,1),
 \]
 and hence $\widetilde f(1/2)\ne 0$ and
 \begin{align*}
  \widetilde g(t) = \frac{1-\frac{3}{2}t}{\widetilde f(1/2)},
     \qquad t\in[1/2,1).
 \end{align*}
By substituting $t:=1/2$ into \eqref{help_Goergens},
 \[
   1-\frac{3}{2}(1-s) = \widetilde f(s)\widetilde g(1/2)
                      = \widetilde f(s)\frac{1}{4\widetilde f(1/2)}, \qquad s\in(0,1/2].
 \]
Then
 \[
   \widetilde f(s) = 4\widetilde f(1/2)\left(1-\frac{3}{2}(1-s)\right),
     \qquad s\in(0,1/2].
 \]
Hence
 \begin{align*}
  \widetilde f(s) \widetilde g(t)
    = 4 \left(1-\frac{3}{2}(1-s)\right) \left(1-\frac{3}{2}t\right)
    = -2 + 6s + 3t - 9st
 \end{align*}
 for $s\in(0,1/2]$ and $t\in[1/2,1)$.
Using \eqref{help_Goergens}, by choosing, e.g., $s:=1/4$ and $t:=2/3$,
 we arrive at a contradiction, since $1-3t(1-s) = -1/2$ and
 $-2 + 6s + 3t - 9st = 0$.
Hence the law of $(M_t)_{t\in[0,1]}$ cannot be the same as the law of
 $(\sqrt{\var(M_t)}\, R_{\tau(t)})_{t\in[0,1]}$ for any monotone function
 $\tau:[0,1]\to \RR$.

Next, we present another short proof for the fact that one cannot find a monotone function $\tau:[0,1]\to \RR$ such that \eqref{help_Goergens_2} holds, communicated to us by Yakov Nikitin.
The right hand side of \eqref{help_Goergens_2} is non-negative for all $s,t\in[0,1]$,
 but one can choose a sufficiently small $\varepsilon\in(0,1)$ such that $\cov(M_\varepsilon,M_{1-\varepsilon})$
 is negative.
Indeed,
 \begin{align*}
  \cov(M_\varepsilon,M_{1-\varepsilon})
   = \varepsilon - \varepsilon(1-\varepsilon)
     - 3\varepsilon^2(1-\varepsilon)^2
   = \varepsilon^2(-3\varepsilon^2+6\varepsilon-2),
 \end{align*}
 which is negative for sufficiently small $\varepsilon\in(0,1)$.

We note that, for a zero area Wiener bridge $(M_t)_{t\in[0,1]}$,
 Deheuvels \cite[Theorem 1.2]{Deh} derived a Karhunen--Lo\`eve expansion,
 Nazarov \cite[Theorem 1]{Naz} investigated small ball probablilities, and G\"orgens \cite[Section 6.1]{Goe} presented a SDE
 with $(M_t)_{t\in[0,1]}$ as a strong solution.
However, up to our knowledge, no integral representation is available
 for $(M_t)_{t\in[0,1]}$, so $(M_t)_{t\in[0,1]}$ does not fit into the framework of Section \ref{section_gen_fram}.

\subsection{Glued Wiener bridge}\label{subsection_glued}
We present another simple counterexample initiated by Helmut Finner.
Namely, if we take two independent Wiener bridges from $0$ to $0$ over $[0,1]$ and over $[1,2]$, respectively, and
 glue them together, then it is a Gauss bridge from $0$ to $0$ over $[0,2]$, \ but
 it cannot be represented as a space-time transformed strictly stationary centered Ornstein-Uhlenbeck process.
Indeed, if $(B^\circ_t)_{t\in[0,1]}$ and $(W^\circ_t)_{t\in[0,1]}$ are independent Wiener bridges from $0$ to $0$
 over $[0,1]$, then for the so-called glued Wiener bridge $G_t:= B^\circ_t\mathbf 1_{[0,1]}(t)
  + W^\circ_{t-1}\mathbf 1_{[1,2]}(t)$, $t\in[0,2]$,
 one cannot find a monotone function $\tau:[0,2]\to \RR$ such that
  \[
   \cov(G_s,G_t) = \sqrt{\var(G_s)}\sqrt{\var(G_t)}\cov(R_{\tau(s)},R_{\tau(t)})
    \qquad \text{for all \ $s,t\in[0,2]$,}
 \]
 where $R=(R_t)_{t\in\RR}$ is a strictly stationary centered Ornstein-Uhlenbeck process with
 $\cov(R_s,R_t) = \ee^{-\frac{\vert t-s\vert}{2}}$, $s,t\in\RR$.
On the contrary, let us suppose that there exists such a function $\tau$.
Then for $s=\frac{1}{2}$ and $t=\frac{3}{2}$, we would have $\cov(G_{1/2},G_{3/2})=0$, and
 \[
   \sqrt{\var(G_{1/2})}\sqrt{\var(G_{3/2})}\cov(R_{\tau(1/2)},R_{\tau(3/2)})
      = \frac{1}{4}\ee^{-\frac{\vert\tau(3/2) - \tau(1/2)\vert}{2}} >0,
 \]
 leading to a contradiction.
Hence the law of $(G_t)_{t\in[0,1]}$ cannot be the same as the law of $(\sqrt{\var(M_t)} R_{\tau(t)})_{t\in[0,1]}$
 for any monotone function $\tau:[0,1]\to\RR$.

\appendix

\section{Proofs}\label{App_proofs}

{\bf First proof of Theorem \ref{Thm_General}.}
Following the notations of Lachout \cite{Lac1}, let us consider a collection of stochastic integrals of
 non-random real functions with respect to a standard Wiener process, namely,
 \[
   \int_0^\infty a_\theta(u)\,\dd B_u, \qquad \theta\in\Theta,
 \]
 where \ $\Theta\subseteq \RR$ \ is a non-empty Borel measurable set, $a_\theta:\RR_+\to\RR$ belongs to $L^2(\RR_+)$  for every
 $\theta\in\Theta$, and $(B_t)_{t\in\RR_+}$ is a standard Wiener process.
Let $f:\Theta\to\RR$  be a Borel measurable function.
By Theorem 4.1 in Lachout \cite{Lac1}, there exists a strictly stationary centered Ornstein-Uhlenbeck process
 $N=(N_t)_{t\in\RR}$ with $\cov(N_s,N_t)=\ee^{-\vert t-s\vert}$, $s,t\in\RR$, such that
 \begin{align}\label{help_Lachout}
    \int_0^\infty a_\theta(u)\,\dd B_u = N(f(\theta)) \qquad \text{ a.s. for all \ $\theta\in\Theta$}
 \end{align}
 if and only if
 \begin{align}\label{help_Lachout2}
    \int_0^\infty a_{\theta_1}(u) a_{\theta_2}(u)\,\dd u = \ee^{-\vert f(\theta_1) - f(\theta_2)\vert}
      \qquad \text{for all \ $\theta_1,\theta_2\in\Theta$.}
 \end{align}
In what follows we apply Lachout's result to our model given in Section \ref{section_gen_fram}.
Namely, let $\Theta:=(0,T)$, and for all $t\in(0,T)$, let $a_t:\RR_+\to\RR$ be given by
 \[
    a_t(u):=\mathbf 1_{[0,t]}(u) \frac{\sigma(u)}{\phi(u)}
                 \left(\int_0^t \frac{\sigma(r)^2}{\phi(r)^2}\,\dd r\right)^{-1/2},
                 \qquad u\in\RR_+.
 \]
Then, for all $t\in(0,T)$, we have $a_t\in L^2(\RR_+)$ and, by \eqref{help_centered},
 \begin{align*}
   \frac{\widetilde Z_t}{\sqrt{\var(Z_t)}}
      = \frac{Z_t - \EE(Z_t)}{\sqrt{\var(Z_t)}}
      =  \left(\int_0^t \frac{\sigma(r)^2}{\phi(r)^2}\,\dd r \right)^{-1/2}
         \int_0^t \frac{\sigma(r)}{\phi(r)}\,\dd B_r
      = \int_0^\infty a_\theta(u)\,\dd B_u.
 \end{align*}
Further, if $t_1\leq t_2$, $t_1,t_2\in(0,T)$, then
 \begin{align*}
  \int_0^\infty a_{t_1}(u)a_{t_2}(u)\,\dd u
     = \frac{\int_0^{t_1} \frac{\sigma(r)^2}{\phi(r)^2}\,\dd r}
            {\left(\int_0^{t_1} \frac{\sigma(r)^2}{\phi(r)^2}\,\dd r \right)^{1/2}
              \left(\int_0^{t_2} \frac{\sigma(r)^2}{\phi(r)^2}\,\dd r \right)^{1/2}}
     = \left(\frac{\int_0^{t_1} \frac{\sigma(r)^2}{\phi(r)^2}\,\dd r}
                   {\int_0^{t_2} \frac{\sigma(r)^2}{\phi(r)^2}\,\dd r}\right)^{1/2}.
 \end{align*}
Let $f:(0,T)\to\RR$ be given by
 \[
   f(t):=\frac{1}{2}\ln\left(\int_0^t \frac{\sigma(r)^2}{\phi(r)^2}\,\dd r \right),
              \qquad t\in(0,T).
 \]
Then
 \begin{align*}
  \ee^{-\vert f(t_1) - f(t_2)\vert}
    = \ee^{f(t_1) - f(t_2)}
    = \left(\frac{\int_0^{t_1} \frac{\sigma(r)^2}{\phi(r)^2}\,\dd r}
                   {\int_0^{t_2} \frac{\sigma(r)^2}{\phi(r)^2}\,\dd r}\right)^{1/2},
                   \qquad t_1\leq t_2, \;\;t_1,t_2\in(0,T).
 \end{align*}
Hence, by Theorem 4.1 in Lachout \cite{Lac1}, there exists a strictly stationary centered Ornstein-Uhlenbeck process
 $N=(N_t)_{t\in\RR}$ with $\cov(N_s,N_t) = \ee^{-\vert t-s\vert}$, $s,t\in\RR$, such that
 \begin{align*}
  \frac{\widetilde Z_t}{\sqrt{\var(Z_t)}}
    = N(f(t))
    \quad \text{a.s. for all \ $t\in(0,T)$.}
 \end{align*}
Hence
 \begin{align*}
   \widetilde Z_t
    = \phi(t)\sqrt{\int_0^t \frac{\sigma(u)^2}{\phi(u)^2}\,\dd u}\;
      N\left(\frac{1}{2}\ln\left(\int_0^t \frac{\sigma(r)^2}{\phi(r)^2}\,\dd r \right)\right)
    \quad \text{a.s. for all \ $t\in(0,T)$.}
 \end{align*}
{By choosing $R_t:=N_{t/2}$, $t\in\RR$, we have}
 \[
    \widetilde Z_t
      = \phi(t) \sqrt{\int_0^t \frac{\sigma(u)^2}{\phi(u)^2}\,\dd u}\;
         R\left(\ln\left(\int_0^t \frac{\sigma(r)^2}{\phi(r)^2}\,\dd r \right)\right)
    \quad \text{a.s. for all \ $t\in(0,T)$,}
 \]
 where $(R_t)_{t\in\RR}$ is a strictly stationary centered Ornstein-Uhlenbeck process
  with $\cov(R_s,R_t) = \ee^{-\frac{\vert t-s\vert}{2}}$, $s,t\in\RR$.
Since $\rat$ is countable, we have
 \[
    \widetilde Z_t
      = \phi(t) \sqrt{\int_0^t \frac{\sigma(u)^2}{\phi(u)^2}\,\dd u}\;
         R\left(\ln\left(\int_0^t \frac{\sigma(r)^2}{\phi(r)^2}\,\dd r \right)\right)
    \quad \text{for all \ $t\in(0,T)\cap\rat$ \ a.s.,}
 \]
 and, since $Z$ and $R$ have continuous sample paths, we have
 \[
    \widetilde Z_t
      = \phi(t) \sqrt{\int_0^t \frac{\sigma(u)^2}{\phi(u)^2}\,\dd u}\;
         R\left(\ln\left(\int_0^t \frac{\sigma(r)^2}{\phi(r)^2}\,\dd r \right)\right)
     \quad \text{for all \ $t\in(0,T)$ \ a.s.}
 \]
This yields \eqref{help_main_transform}, since $\PP(\widetilde Z_t \to \widetilde Z_0=0 \ \text{as $t\downarrow 0$})=1$
  following from the facts that $(Z_t)_{t\in[0,T)}$ has continuous sample paths almost surely and $\EE(Z_t) = \phi(t)\xi
 + \phi(t)\int_0^t \frac{\psi(u)}{\phi(u)}\,\dd u$, $t\in[0,T)$, is continuous.
Finally, Equation \eqref{rep_closed} readily follows by \eqref{help_main_transform} and $\PP(\lim_{t\uparrow T} \widetilde Z_t = 0)=1$.
\proofend

\medskip

{\bf Second proof of Theorem \ref{Thm_General}.}
The proof consists of two parts: first we check that the pathwise unique strong solution of the SDE \eqref{SDE_general}
 can be represented as a space-time transformed standard Wiener process, and then we use Lamperti transformation
 recalled in the introduction.
Namely, by Dambis, Dubins and Schwarz lemma (see, e.g., Revuz and Yor \cite[Chapter V, Theorems 1.6 and 1.7]{RevYor}
 or Karatzas and Shreve \cite[Theorem 3.4.6 and Problem 3.4.7]{KarShr}),
 there exists a standard Wiener process $(W_t)_{t\in\RR_+}$ (possibly on an enlargement of the original
 probability space and stopped at $\lim_{t\uparrow T}\int_0^t \frac{\sigma(u)^2}{\phi(u)^2}\,\dd u$) such that
 \begin{align}\label{help_main_proof1}
 \widetilde Z_t
    = \phi(t)\int_0^t \frac{\sigma(u)}{\phi(u)}\,\dd B_u
    = \phi(t)W\left(\int_0^t \frac{\sigma(u)^2}{\phi(u)^2}\,\dd u\right),
   \qquad \forall \; t\in[0,T)\quad \text{a.s.}
 \end{align}
Indeed, $\int_0^t \frac{\sigma(u)}{\phi(u)}\,\dd B_u$, $t\in[0,T)$, is a continuous $L^2$-martingale starting at $0$,
 since $\int_0^t \frac{\sigma(u)^2}{\phi(u)^2}\,\dd u<\infty$ for all $t\in[0,T)$,
 and we note that even if $\lim_{t\uparrow T}\int_0^t \frac{\sigma(u)^2}{\phi(u)^2}\,\dd u = \infty$
 does not hold, one can apply Dambis, Dubins and Schwarz lemma.
Let
 \[
    R_t:= \ee^{-\frac{t}{2}} W_{\ee^t},\qquad t\in\RR,
 \]
 yielding $W_t = \sqrt{t} R_{\ln(t)}$, $t>0$.
Then, as it was recalled in the Introduction, $R$ is a strictly stationary centered Ornstein-Uhlenbeck process
 with $\cov(R_s,R_t) = \ee^{-\frac{\vert t-s\vert}{2}}$, $s,t\in\RR$,
 and, by \eqref{help_main_proof1},
 \[
    \widetilde Z_t = \phi(t)\sqrt{\int_0^t \frac{\sigma(u)^2}{\phi(u)^2}\,\dd u}\,
                           \,R\left(\ln\left( \int_0^t \frac{\sigma(u)^2}{\phi(u)^2}\,\dd u \right)\right)
 \]
 for all $t\in(0,T)$ almost surely.
Further,
 \begin{align}\label{help_main_transform2}
   \begin{split}
   &\lim_{t\downarrow 0}
    \phi(t)\sqrt{\int_0^t \frac{\sigma(u)^2}{\phi(u)^2}\,\dd u}\,
     \,R\left(\ln\left( \int_0^t \frac{\sigma(u)^2}{\phi(u)^2}\,\dd u \right)\right) \\
   &\qquad  \qquad = \lim_{t\downarrow 0} \phi(t)W\left(\int_0^t \frac{\sigma(u)^2}{\phi(u)^2}\,\dd u \right)
            = \phi(0)W_0
            =0 \qquad \text{a.s.,}
   \end{split}
 \end{align}
 yielding \eqref{help_main_transform}, as desired.
Finally, Equation \eqref{rep_closed} readily follows by \eqref{help_main_transform} and $\PP(\lim_{t\uparrow T} \widetilde Z_t = 0)=1$.
\proofend

\medskip

{\bf Proof of Proposition \ref{Prop_General_reversed}.}
Since for all \ $t\in(0,T)$,
 \begin{align*}
  \phi(t)\sqrt{\int_0^t \frac{\sigma(u)^2}{\phi(u)^2}\,\dd u}
     = \frac{a(t)}{\ee^{b(t)/2}} \sqrt{\int_0^t  b'(u)\ee^{b(u)}\,\dd u}
     = \frac{a(t)}{\ee^{b(t)/2}} \sqrt{\ee^{b(t)} - \lim_{t\downarrow 0} \ee^{b(t)}}
     = a(t)
 \end{align*}
 and
 \begin{align*}
  \ln\left(\int_0^t \frac{\sigma(u)^2}{\phi(u)^2}\,\dd u\right) = \ln(\ee^{b(t)}) = b(t),
 \end{align*}
 Theorem \ref{Thm_General} yields the statement.
\proofend

\medskip

{\bf Proof of Proposition \ref{Prop_General}.}
We need to check that
 \[
   \lim_{t\uparrow T} \widetilde Z_t
     = \lim_{t\uparrow T} \phi(t) W\left(\int_0^t \frac{\sigma(u)^2}{\phi(u)^2}\,\dd u \right)
     = 0
 \]
 almost surely, where $(W_t)_{t\in\RR_+}$ is the standard Wiener process appearing in the second proof of Theorem \ref{Thm_General}.
If $\int_0^T  \frac{\sigma(u)^2}{\phi(u)^2}\,\dd u := \lim_{t\uparrow T} \int_0^t \frac{\sigma(u)^2}{\phi(u)^2}\,\dd u \in\RR_+$,
 then
 \begin{align*}
    \lim_{t\uparrow T} \phi(t) W\left(\int_0^t \frac{\sigma(u)^2}{\phi(u)^2}\,\dd u \right)
      &= \left(\lim_{t\uparrow T} \phi(t) \right) W\left(\int_0^T  \frac{\sigma(u)^2}{\phi(u)^2}\,\dd u \right)\\
      &= 0 \cdot W\left(\int_0^T  \frac{\sigma(u)^2}{\phi(u)^2}\,\dd u \right) = 0
      \qquad \text{a.s.}
 \end{align*}
If $\int_0^T  \frac{\sigma(u)^2}{\phi(u)^2}\,\dd u=\infty$, then
 \[
    \lim_{t\uparrow T} \phi(t) W\left(\int_0^t \frac{\sigma(u)^2}{\phi(u)^2}\,\dd u \right)
     = \lim_{t\uparrow T}
         \phi(t) \left( \int_0^t \frac{\sigma(u)^2}{\phi(u)^2}\,\dd u \right)^{\frac{1}{2}+\varepsilon}
         \frac{W\left(\int_0^t \frac{\sigma(u)^2}{\phi(u)^2}\,\dd u \right)}
               {\left( \int_0^t \frac{\sigma(u)^2}{\phi(u)^2}\,\dd u \right)^{\frac{1}{2}+\varepsilon}}
     =0
 \]
 almost surely, where we used that
 \[
   \lim_{s\to\infty} \frac{W_s}{s^{\frac{1}{2} + \eta}} = 0 \qquad \text{a.s. for all \ $\eta>0$,}
 \]
 which follows by the law of iterated logarithm for a standard Wiener process
 (see, e.g., Revuz and Yor \cite[Chapter II, Corollary 1.12]{RevYor}).
Indeed,
 \begin{align*}
 \frac{W_s}{s^{\frac{1}{2}+\eta}}
   = \frac{\sqrt{2s\ln(\ln(s))}}{s^{\frac{1}{2}+\eta}}
     \frac{W_s}{\sqrt{2s\ln(\ln(s))}}
   = \frac{\sqrt{2\ln(\ln(s))}}{s^\eta}
     \frac{W_s}{\sqrt{2s\ln(\ln(s))}},
   \qquad s>\ee,
 \end{align*}
 where, by the law of iterated logarithm for a standard Wiener process,
 \[
      \left(\frac{W_s}{\sqrt{2s\ln(\ln(s))}}\right)_{s>\ee}
 \]
 is bounded almost surely, and, by $\cL$'Hospital's rule,
 \[
  \lim_{s\to\infty} \frac{\sqrt{2\ln(\ln(s))}}{s^\eta}
   = \lim_{s\to\infty} \frac{1}{\sqrt{2}\eta s^\eta \ln(s) \sqrt{\ln(\ln(s))}}
   = 0.
 \]
\proofend

\section*{Acknowledgements}
\noindent We are grateful to Yakov Nikitin for bringing the papers of Deheuvels \cite{Deh} and Nazarov \cite{Naz} to our attention,
 and for giving a short explanation for the fact that a zero area Wiener bridge cannot be represented as a space-time scaled
 strictly stationary centered Ornstein-Uhlenbeck process, presented in Subsection \ref{Zero area Wiener bridge}.
We are also grateful to Helmut Finner for initiating the counterexample presented in Subsection \ref{subsection_glued}.

\bibliographystyle{plain}

\end{document}